\documentclass[reqno]{amsart}

\usepackage[latin1]{inputenc} 
\usepackage{latexsym,amsfonts,amssymb,amsmath,euscript}

\usepackage{graphicx} 
\usepackage{color,fancybox}
\usepackage{color}
\usepackage{graphicx} 

\usepackage{enumerate}

\newcommand{\R}
 {\mathbb{R}}

 \newcommand{\E}
 {\mathcal{E}}

\def\vareps{\varepsilon}
\newcommand{\ffrac}{\displaystyle\frac}

\DeclareMathOperator{\dist}{dist}

\begin{document}

\title[Elliptic equations with terms concentrating on the oscillatory boundary]{Semilinear elliptic equations in thin regions \\ with terms concentrating on oscillatory boundaries}

\author[J. M. Arrieta, A. Nogueira and M. C. Pereira]
{Jos\'e M. Arrieta$^*$, Ariadne Nogueira$^\diamond$ and Marcone C. Pereira$^\dagger$}

\address{Jos\'e M. Arrieta
\hfill\break\indent Dpto. de An\'alisis Matem\'atico y Matem{\'a}tica Aplicada, Fac. Ciencias Matem\'aticas
\hfill\break\indent  Universidad Complutense de Madrid, 28040 Madrid, Spain. 
\hfill\break\indent and
\hfill\break\indent Instituto de Ciencias Matem\'aticas, CSIC-UAM-UC3M-UCM 
\hfill\break\indent C/ Nicol\'as Cabrera 13-15, 28049 Madrid, Spain }
\email{{\tt arrieta@mat.ucm.es} }


\address{Ariadne Nogueira
\hfill\break\indent Dpto. de Matem{\'a}tica, IME,
Universidade de S\~ao Paulo, \hfill\break\indent Rua do Mat\~ao 1010, 
S\~ao Paulo - SP, Brazil. } \email{{\tt ariadnen@ime.usp.br} }

\address{Marcone C. Pereira
\hfill\break\indent Dpto. de Matem{\'a}tica Aplicada, IME,
Universidade de S\~ao Paulo, \hfill\break\indent Rua do Mat\~ao 1010, 
S\~ao Paulo - SP, Brazil. } \email{{\tt marcone@ime.usp.br} 
}

\date{}

\subjclass[2010]{34B15, 35J75, 35J91} 
\keywords{Semilinear elliptic equations, singular elliptic equations, upper semicontinuity, lower semicontinuity, thin domains, concentrating terms.} 

\begin{abstract}

In this work we study the behavior of a family of solutions of a semilinear elliptic equation, with homogeneous Neumann boundary condition, posed in a two-dimensional oscillating thin region with reaction terms concentrated in a neighborhood of the oscillatory boundary. 
Our main result is concerned with the upper and lower semicontinuity of the set of solutions. 
We show that the solutions of our perturbed equation can be approximated with ones of a one-dimensional equation, which also captures the effects of all relevant physical processes that take place in the original problem.

\end{abstract}

\maketitle
\numberwithin{equation}{section}
\newtheorem{theorem}{Theorem}[section]
\newtheorem{lemma}[theorem]{Lemma}
\newtheorem{corollary}[theorem]{Corollary}
\newtheorem{proposition}[theorem]{Proposition}
\newtheorem{remark}[theorem]{Remark}
\newtheorem{definition}[theorem]{Definition}
\allowdisplaybreaks

\section{Introduction}

In this paper we investigate the behavior of a family of solutions given by a semilinear elliptic equation, with homogeneous Neumann boundary condition, defined in a two-dimensional oscillating thin region $R_\vareps$ with reaction terms concentrated in a neighborhood $o_\vareps$ of the oscillatory boundary of $R_\varepsilon$. We deal with an elliptic reaction-diffusion equation posed in the bounded open set $R_\vareps$ which degenerates into a line segment as the positive parameter $\vareps$ goes to zero. Also, we assume that the reaction of the model only occur in a narrow strip $o_\vareps$ close to the border, which also can present high oscillatory structure.
See Figure \ref{fig:dominio_inicial} below which illustrates the open region $R_\varepsilon$ and the narrow neighborhood $o_\varepsilon$ mentioned here. 

\begin{figure}[h]
	\centering\includegraphics[scale=0.5]{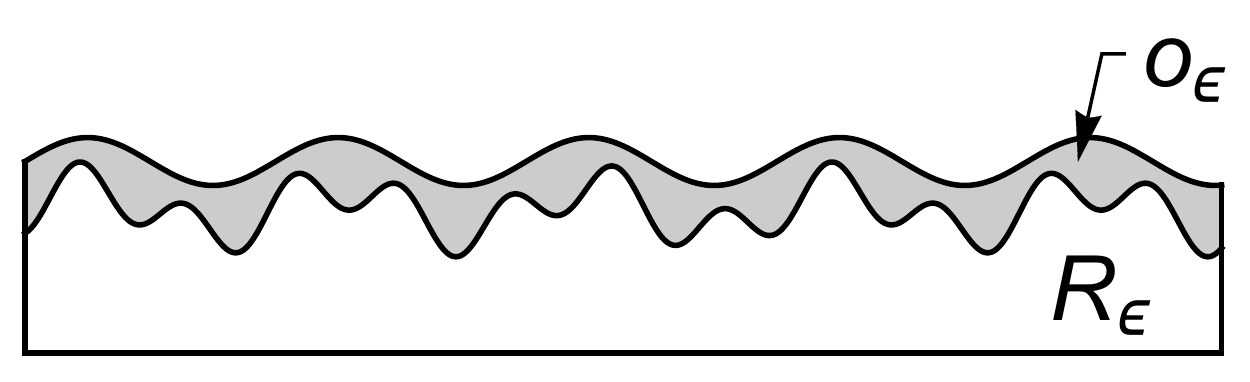}
	\caption{The thin domain $R_\vareps$ and the strip $o_\vareps$ where reactions take place.} \label{fig:dominio_inicial}
\end{figure}

Our main result is that the family of solutions are upper and lower semicontinuous at $\vareps=0$. Indeed, we show that the starting singular equation defined in the two-dimensional region can be approximated with one which is a one-dimensional regular equation, which captures the effects of all relevant physical processes that take place in the original problem. Therefore, the limit equation will preserve features of the original system, giving conditions to access the qualitative behavior of the modeled problem in a simpler way.

Let us recall that elliptic boundary value problem models diffusion and interactions among agents which can be cells, amount of chemicals or biological organisms.
Thus, we are supposing here that the agents are located in an extremely thin region with reactions taking place just in a small neighborhood of the border.
It is worth noting that our model includes the possibility that the thin region as well the narrow neighborhood present high oscillatory behavior, modeling complex regions of interactions.

Potential applications of our results can be seen for instance in \cite{S,MS,LC,CLS,PB,PS,Apl3,MTA} where theoretical and practical aspects of mathematical modeling and applications are investigated. 
The fields mentioned are such as lubrication, nanotechnology, fluid-structure interaction mechanism in vascular dynamics and management and control of aquatic ecological systems,  where one can find localized concentrations in connection with boundary complexity in thin channels. 

There are several works dealing with partial differential equations posed in thin domains. 
We first mention the pioneering works \cite{HaleR,R}, as well the subsequent papers \cite{E,PR01,PRR}, where the authors investigate the asymptotic behavior of dynamical systems given by a class of semilinear parabolic equations in thin domains of $\R^n$, $n \geq 2$.
We also cite \cite{PS3, Ricardo} where the $p$-Laplacian problem in thin regions is considered, and \cite{MP}, which studies a linear elliptic problem in perforated thin domains with rapidly varying thickness.
In \cite{BGG} the authors consider nonlinear monotone problems in a multidomain with a highly oscillating boundary.
In \cite{arr_marcone_artigo_base,AP2,AP3,AVP,MR2, mcp2,AM,AMJM} and references therein, we have recently studied many classes of oscillating thin regions for elliptic and parabolic equations with Neumann boundary conditions, discussing limit problems and convergence properties. For nonlocal equations in thin structures we also mention \cite{GH,GH2,MJ,MJ2}.

On the other hand, there are many works in the literature concerned with singular elliptic and parabolic problems  featuring potential and reactions terms concentrated in a small neighborhood of a portion of the boundary. In fixed bounded domains, we cite the pioneering works \cite{arrieta,anibal,anibal2}. In regions presenting oscillatory behavior, we mention the recent ones \cite{GSim,GSim2}. In \cite{APP,APP2} we also have studied problems allowing narrow strips with oscillatory border in fixed bounded open sets. 

Our main goal here is to discuss a model combining these both singular situations (the thin domain problem and concentrated reactions) in a more general framework. For this, we generalize \cite{marcone_saulo} adapting methods and techniques  developed in \cite{arrieta} to deal with concentrated integrals getting appropriate estimates. Then, we can pass to the limit in our model obtaining its asymptotic behavior at $\epsilon = 0$.  

The paper is organized as follows: in Section \ref{resultado}, we set our assumptions, notations and state the main result concerning to the upper and lower semicontinuity of the set of solutions. In Section \ref{espacos}, we introduce our functional setting, and obtain results which allow us to estimate the concentrated integrals.
In Section \ref{nonlinear}, we deal with nonlinear maps related with the nonlinear reaction terms of the equation, and in Section \ref{mainsec}, we show our main result getting the asymptotic behavior of the solutions at $\epsilon=0$.

\section{Assumptions, notations and main result}
\label{resultado}

Let us consider the following semilinear elliptic equation with homogeneous Neumann boundary conditions
\begin{align}
\label{original}
\begin{cases}
-\Delta v^\vareps + v^\vareps = \ffrac{1}{\vareps}\chi^{o_\vareps}f(v^\vareps) \ & \text{ in } R_\vareps \\ 
\ffrac{\partial v^\vareps}{\partial \nu^\vareps} = 0 \ & \text{ on } \partial R_\vareps
\end{cases}
\end{align}
where for each $\vareps>0$, $R_\vareps \subset \R^2$ is an oscillating thin domain given by
$$
R_\vareps = \{(x,y); \ 0<x<1, \ 0<y<\varepsilon g(x/\vareps)\}. 
$$
The vector $\nu^\vareps = (\nu^\vareps_1,\nu^\vareps_2)$ denotes the unit outward normal vector to the boundary $\partial R_\vareps$, 
$\partial/\partial \nu^\vareps$ is the normal derivative, and $\chi^{o_\vareps}$ is the characteristic function of the set $o_\vareps$ defined by  
\begin{equation*}
o_\vareps = \{(x,y); \ 0<x<1, \ \varepsilon(g(x/\vareps)-\vareps h(x/\vareps^\beta))<y<\varepsilon g(x/\vareps)\}.
\end{equation*}
We assume 
\begin{enumerate}[$(i)$]
	\item $\beta>0$;
	\item the nonlinearity $f: \R\to\R$ is a $\mathcal{C}^2$ function;
	\item the functions $g,h : (0,1) \to \R$ are positive, $L_g$ and $L_h$-periodic, respectively, and possess $g_0,g_1,h_0,h_1\in\R$ such that 
	\begin{equation*}
	0<g_0\leq g(x)\leq g_1<\infty, \quad 0\leq h_0\leq h(x)\leq h_1<\infty, \quad \forall x\in (0,1);
	\end{equation*}
	\item $g$ has bounded derivative.
\end{enumerate}

\begin{remark}\label{muh}
	Notice that, calling $g_\vareps(x)=g(x/\vareps)$ and $h_\vareps(x)=h(x/\vareps^\beta)$, it follows from \cite[Theorem 2.6]{doina} that there exist $\mu_g$, $\mu_h\in\R$ such that
	\begin{equation*}
	g_\vareps\stackrel{*}{\rightharpoonup}\mu_g=\ffrac{1}{L_g}\int_0^{L_g}g(s)ds \quad \text{ and } \quad h_\vareps\stackrel{*}{\rightharpoonup}\mu_h=\ffrac{1}{L_h}\int_0^{L_h}h(s)ds \ \text{ in } L^\infty(0,1).
	\end{equation*}
	The constants $\mu_g$ and $\mu_h$ are the average of the periodic functions $g$ and $h$ respectively. 
\end{remark}

Let us emphasize that $R_\vareps \subset (0,1) \times (0, \vareps g_1)$ is a two-dimensional thin region with oscillatory boundary 
which degenerates to the unit interval as $\vareps \to 0$. 
Also,  $o_\vareps \subset R_\vareps$ represents an oscillating $\varepsilon$-neighborhood to the upper boundary of $R_{\vareps}$ where the reaction term takes place. 

Notice that here, we are in agreement with \cite{arrieta}. We combine the characteristic function $\chi^{o_{\vareps}}$ and the positive parameter $\vareps$ in order to set concentration of reactions on the small strip $o_\vareps\subset R_\vareps$ through the term
\begin{equation*}
\ffrac{1}{\vareps}\chi^{o_\vareps}\in L^\infty(R_\vareps).
\end{equation*}

We will show that, in a certain functional setting, the family of solutions from the perturbed problem \eqref{original} converges to a solution of a one-dimensional equation of the same type, with homogeneous Neumann boundary condition, capturing the variable profile of the domain $R_\vareps$ as well as the oscillatory behavior of the neighborhood $o_\vareps$. Indeed, we obtain the following limit problem 
\begin{equation}\label{limite}
\begin{cases}
-q_0u_{xx}+u=f_0(u) \quad \text{ in }  (0,1), \\ u_x(0)=u_x(1)=0
\end{cases}
\end{equation}
with \begin{equation}\label{q_0}
q_0=\ffrac{1}{|Y^*|}\int_{Y^*}\left\{ 1-\ffrac{\partial X}{\partial y_1}(y_1,y_2) \right\} dy_1dy_2 \quad \text{ and } \quad f_0(\cdot)=\ffrac{L_g}{|Y^*|}\mu_h f(\cdot).
\end{equation}
The function $X$ is the unique solution of the auxiliary problem
\begin{equation*}
\begin{cases}
-\Delta X=0 \text { in } Y^* \\
\frac{\partial X}{\partial N}=0 \text{ in } B_2 \\
\frac{\partial X}{\partial N}=N_1 \text{ in } B_1 \\ X \text{ is $L_g$-periodic in }y_1 \\ \int_{Y^*} Xdy_1dy_2=0
\end{cases}
\end{equation*} 
where $Y^*$ given by
\begin{equation*}
Y^*=\{(y_1,y_2)\in\R^2; \ 0<y_1<L_g, 0<y_2<g(y_1) \}
\end{equation*}
is the representative cell of the thin region $R_\vareps$. The vector $N=(N_1,N_2)$ is the outward normal vector to the boundary $\partial Y^*$ with $B_1$ and $B_2$ denoting the upper and lower boundary of $\partial Y^*$ respectively.

Notice that the diffusion coefficient $q_0$, usually called homogenized coefficient, exhibit the effect of the geometry and the oscillatory behavior of the thin region. 
On the other hand, nonlinearity $f_0$ captures the influence of the concentration neighborhood on the reaction term $f$.
The limit problem \eqref{limite} is often called homogenized equation.

It is worth noting that the results obtained here generalize the ones from \cite{marcone_saulo} since the thin domain analyzed there does not exhibit any oscillatory behavior.
Furthermore, we emphasize that our task is not easy here. In order to accomplish our goal, we have to be able to estimate the solutions in very small neighborhoods of the oscillatory boundary in such way that we can pass to the limit at $\vareps=0$.

The solutions of our problem are defined in open sets which varies with respect to parameter $\vareps>0$. 
Thus, the first step in our analysis is to set an approach in order to face this domain perturbation problem. 
Here we adopt the same strategy used, for instance, in \cite{arr_marcone_artigo_base}.
We rescale the thin region $R_\vareps$ keeping the $x$-coordinate and multiplying the values of $y$ by a factor $1/\vareps$ avoiding the thin domain situation.  
Performing this change of variable, we obtain the following problem:
\begin{equation} \label{inicio}
\left\{
\begin{gathered}
-\ffrac{\partial^2 u^\vareps}{\partial x_1^2}-\ffrac{1}{\vareps^2}\ffrac{\partial^2 u^\vareps}{\partial x_2^2}+u^\vareps = \ffrac{1}{\vareps}\chi^{\theta_\vareps}f(u^\vareps) \quad  \text { in } \Omega_\vareps, \\
\ffrac{\partial u^\vareps}{\partial x_1}N_1^\vareps + \ffrac{1}{\vareps^2}\ffrac{\partial u^\vareps}{\partial x_2}N_2^\vareps=0  \quad \text{ on } \partial \Omega_\vareps,
\end{gathered}
\right.
\end{equation}
where $N^\vareps = (N_1^\vareps,N_2^\vareps)$ is the outward normal vector to the boundary $\partial \Omega_\vareps$,
\begin{equation*}\label{dom_fino}
\begin{gathered}
\Omega_\vareps = \{(x_1,x_2); \ 0<x_1<1, \ 0<x_2<g_\varepsilon(x_1)\} \quad \textrm{ and } \\
\theta_\vareps = \{(x_1,x_2); \ 0<x_1<1, \ g_\varepsilon(x_1) - \vareps h_\vareps(x_1)<x_2<g_\varepsilon(x_1)\}.
\end{gathered}
\end{equation*}

\begin{figure}[h]
	\centering\includegraphics[scale=0.3]{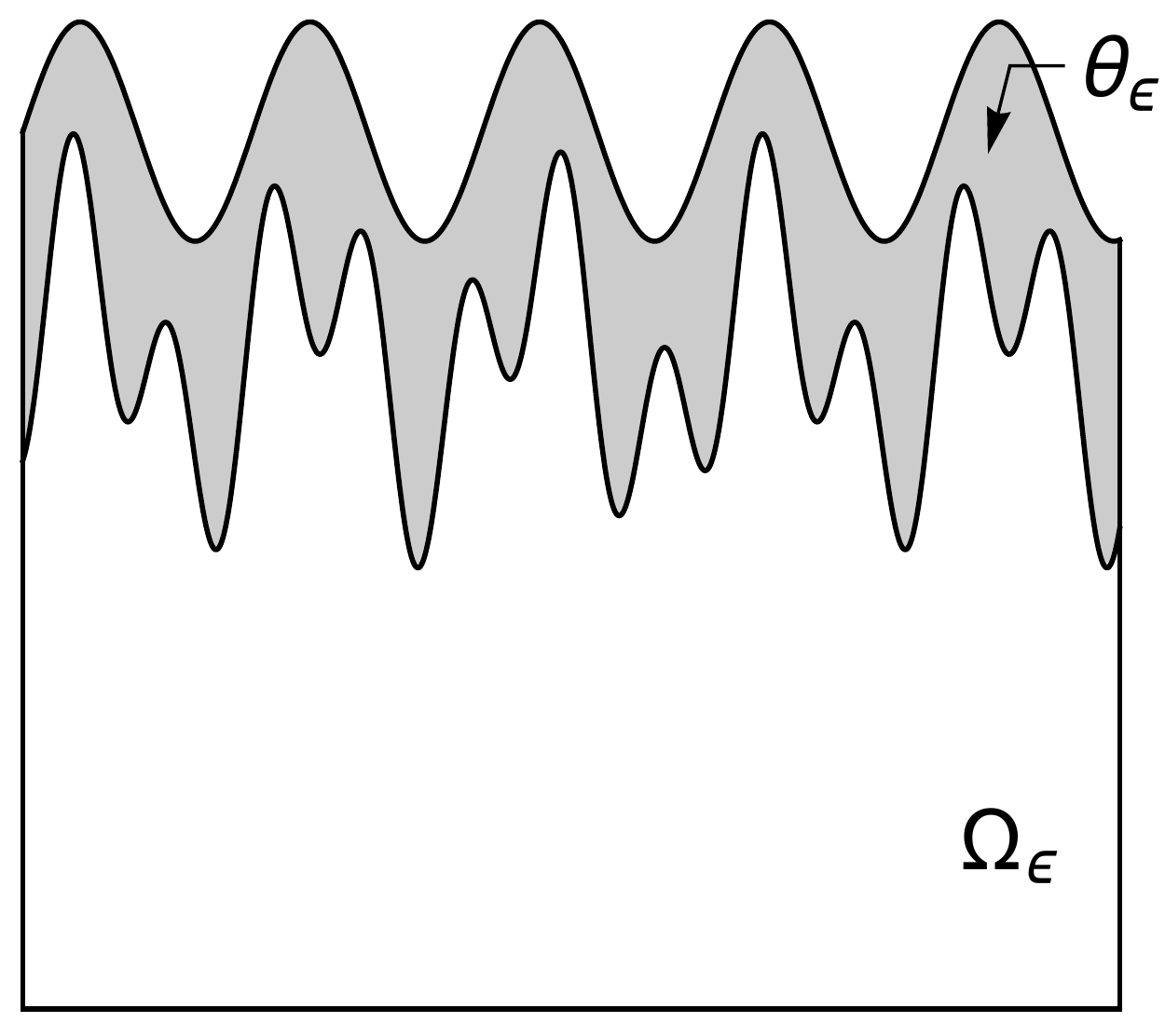}
	\caption{The modified domain $\Omega_\vareps$ and the neighborhood $\theta_\vareps$.} \label{fig:dominio_modificado}
\end{figure}

It is not difficult to see that problems \eqref{original} and \eqref{inicio} are equivalent.
In some sense, we have rescaled the neighborhood $o_\varepsilon$ into the strip $\theta_\varepsilon \subset \Omega_\vareps$ and substituted the thin region $R_\epsilon$ by the oscillating domain $\Omega_\vareps$, at a cost of introducing a very strong diffusion mechanism in the $x_2$-direction given by the factor $1/\vareps^2$. This will make the solutions from \eqref{inicio} to become more and more homogeneous in this direction as $\vareps$ goes to zero. 
In this way, the limit solution will not  depend on $x_2$, and therefore, the limit equation will be one dimensional. Notice that is in full agreement with the intuitive idea that a partial differential equation posed in a thin domain should approach one defined in a line segment.

In order to obtain our convergence results, we have to compare functions defined in different functional spaces. Fixed $1/2<s<1$, we consider the Lebesgue-Bochner spaces 
$$X = L^2(0,1;H^s(0,g_1)) \quad \textrm{ and } \quad X_\vareps = L^2(0,1;H^s(0,g_\vareps(x_1))),$$ 
which will be discussed in Section \ref{espacos}, as well as the Hilbert space $X_0=L^2(0,1)$ with the norm given by $$\|u\|_{X_0}=\sqrt{\mu_g}\|u\|_{L^2(0,1)}.$$
The perturbed problem \eqref{original} will be set in $X_\vareps$, and the limit equation \eqref{limite} in $X_0$.

Since $X_0\subset X_\vareps$, we can consider the operator
\begin{align}\label{edef}
E_\vareps : X_0 \to& X_\vareps \nonumber \\ u(x_1)&\mapsto (E_\vareps u)(x_1,x_2)=u(x_1)
\end{align}
which satisfies 
\begin{eqnarray*}
\|E_\vareps u\|^2_{X_\vareps} & = & \int_{0}^{1}\|(E_\vareps u)(x_1,\cdot)\|_{H^s(0,g_\vareps(x_1))}^2dx_1=\int_{0}^{1}g_\vareps(x_1)|u(x_1)|^2dx_1 \\
& \to & \mu_g\|u\|_{L^2(0,1)}=\|u\|^2_{X_0}, \  \textrm{ as } \varepsilon \to 0.
\end{eqnarray*}
As in \cite{CP}, we obtain an appropriate way to compare solutions from \eqref{inicio} and \eqref{limite}.

\begin{definition}
We say that $\{ u^\vareps \}_{\vareps >0} \subset X_\vareps$ $E$-converges to a function $u\in X_0$, if $\|u^\vareps-E_\vareps u\|_{X_\vareps}\to0$ as $\vareps\to0$, where $E_\vareps$ is given by \eqref{edef}.
It is denoted by $u_\vareps\xrightarrow{E}u$.
\end{definition}
	
This notion of convergence can be also extended to sets in the following manner: let $J_\vareps$ be a family of sets in $X_\vareps$. We say that $J_\varepsilon\subset X_\vareps$ is
\begin{enumerate}[(i)]
	\item upper semicontinuous at $\vareps=0$, if $\dist_H(J_\vareps, E_\vareps J_0)\xrightarrow{\vareps\to0}0$;
	\item lower semicontinuous at $\vareps=0$, if $\dist_H(E_\vareps J_0,J_\vareps)\xrightarrow{\vareps\to0}0$.
\end{enumerate}
Here, $\dist_H(A,B)$ denotes the Hausdorff semi-distance given by 
\begin{equation*}\label{haus}
\dist_H(A,B)=\sup_{x\in A} \inf_{y\in B}\|x-y\|_{X_\vareps}.
\end{equation*}

\begin{remark}\label{conv_cont} Also, the following characterizations are very useful:
	\begin{enumerate}[(i)]
		\item The family $\{J_\vareps\}$ is upper semicontinuous at $\vareps=0$ if every sequence $\{u_\vareps\}$, with $u_\vareps\in J_\vareps$ and $\vareps\to0$, has a subsequence $E$-convergent to an element of $J_0$;
		\item The family $\{J_\vareps\}$ is lower semicontinuous at $\vareps=0$ if $J_0$ is compact and for all $u\in J_0$ exists a sequence $\{u_\vareps\}$, with $u_\vareps\in J_\vareps$ and $\vareps\to0$, such that $u_\vareps\xrightarrow{E}u$.
	\end{enumerate}
\end{remark}

Finally, let us consider, for $0<\vareps\leq\vareps_0$, the folowing sets given by problem \eqref{inicio} 
$$\E_{\vareps,R}=\{u^\vareps\in H^1(\Omega_\vareps); \ u^\vareps \text{ is a solution of } \eqref{inicio} \text{ and } \|u^\vareps\|_{L^\infty(\Omega_{\vareps})}\leq R\}$$
and
$$\E_{0,R}=\{u\in H^1(0,1); \ u \text{ is a solution of } \eqref{limite} \text{ and } \|u\|_{L^\infty(0,1)}\leq R\}.$$
Now, we state our main result which concerns on upper and lower semicontinuity of the set $\E_{\vareps,R}$ at $\vareps=0$. 
\begin{theorem}\label{teo_main}
	\begin{enumerate}[(a)]
		\item For any sequence $u_\vareps\in\E_{\vareps, R}$, with $\vareps\to0$, there is a subsequence (also denoted by $u_\vareps$) and $u_0\in \E_{0,R}$ such that $u_\vareps\xrightarrow{E}u_0$ in $X_\varepsilon$ when $\vareps\to0$.
		\item For any hyperbolic equilibrium point  $u_0\in\E_{0,R}$, there is sequence  $u_\vareps\in\E_{\vareps,R}$ such that $u_\vareps\xrightarrow{E}u_0$ in $X_\varepsilon$ as $\vareps\to0$.
	\end{enumerate}
\end{theorem}

\begin{remark}
Recall that a solution $u$ of a boundary value problem is hyperbolic if $\lambda=0$ is not an eigenvalue of the linearized problem around $u$. 
In other words, $u \in \E_{0,R}$ is hyperbolic if $\lambda=0$ is not an eigenvalue of the eigenvalue problem 
$$
\left\{ 
\begin{gathered}
- q_0 v_{xx} + v = \partial_u f_0(u) v + \lambda v \quad \textrm{ in } (0,1) \\       
v_x(0)=v_x(1)=0
\end{gathered} \right. 
$$
where $q_0$ and $f_0$ are defined in \eqref{q_0}.
\end{remark}

\begin{remark} \label{limit_f}
Since we are concerned with solutions which are  uniformly bounded  in $L^\infty(\Omega)$, we may take $f$ of class $\mathcal{C}^2$ bounded with bounded derivatives. In fact, we may perform a cut-off in $f$ outside the region $|u| \leq R$  without modifying any of these solutions (see for instance \cite[Remark 2.2]{arrieta_simone_lips} or \cite[Remark 2.2]{dumb_1}).
\end{remark}

\begin{remark}
The assertions $(a)$ and $(b)$ in Theorem \ref{teo_main} respectively mean {upper} and {lower} semicontinuity of the equilibria set to the parabolic problem associated with \eqref{inicio} at $\epsilon=0$.  
\end{remark}

\section{Functional spaces and concentrated integrals}
\label{espacos}

In this section, we first establish the functional spaces used to analyze the concentrated integrals. Next, we perform some estimates in such functional spaces.

\begin{definition}
	Let $s=m+\sigma>0$, with $m=0,1,\ldots$, $0\leq \sigma<1$ and let $O\subset R^n$ a domain. For $1\leq p<\infty$, we call fractional Sobolev space $W^{s,p}(O)$ the space of functions $u$ such that
	\begin{enumerate}[(i)]
		\item $u\in W^{m,p}(O)$ if $\sigma=0$,
		\item $u\in W^{m,p}(O)$ and 
		\begin{equation*}
		\iint_{O\times O}\ffrac{|\partial^\alpha u(x)-\partial^\alpha u(y)|^p}{|x-y|^{n+\sigma p}}dxdy<\infty
		\end{equation*}
		if $\sigma>0$. 
	\end{enumerate}
	
	The norm in $W^{s,p}(O)$, that makes it Banach, is:
	\begin{equation*}
	\|u\|_{W^{m,p}(O)}^p=\sum_{|\alpha|\leq m}\int_{O}|\partial^\alpha u|^pdx \ \text{ in the case }(i)
	\end{equation*}
	and
	\begin{equation*}
	\|u\|_{W^{s,p}(O)}^p=\|u\|_{W^{m,p}(O)}^p+\sum_{|\alpha|=m} \iint_{O\times O}\ffrac{|\partial^\alpha u(x)-\partial^\alpha u(y)|^p}{|x-y|^{n+\sigma p}}dxdy\ \text{ in the case }(ii).
	\end{equation*}
	
	Furthermore, if $p=2$ we call it $H^s(O)$ and it is a Hilbert space.\\
\end{definition}

%
%

Now let us follow \cite{leb_boch} to introduce what we call Lebesgue and Sobolev-Bochner generalized spaces. They are a natural generalization to Lebesgue and Sobolev spaces using Bochner integrals.
The usual Lebesgue and Sobolev-Bochner spaces may be found for instance in \cite{doina,cazenave}. \\
%
%
%

Let us consider a function $G:(0,1)\to \R$ satisfying that there exist $0<G_0\leq G_1$ with $G_0\leq G(x)\leq G_1$.

\begin{definition}
	Let us consider a function $G:(0,1)\to \R$ satisfying $0<G_0\leq G(x)\leq G_1$ for some constants  $0<G_0\leq G_1$. 
Let $1\leq p\leq\infty$, $1\leq q<\infty$. The Lebesgue-Bochner generalized spaces, denoted by $L^p(0,1;L^q(0,G(x_1)))$, are defined by
	\begin{equation*}
	L^p(0,1;L^q(0,G(x_1))):=\{u:\Omega_{\vareps}\to\R \text{ measurable}; u(x_1,\cdot)\in L^q(0,G(x_1)) \text{ for almost every } x\in\Omega\}
	\end{equation*}
	and they are Banach spaces with the norm
	\begin{align*}
	\|u\|_{L^p(0,1;L^q(0,G(x_1)))}=\begin{cases}
	\left(\displaystyle \int_0^1\|u(x_1,\cdot)\|_{L^q(0,G(x_1))}^p dx_1\right)^{1/p}, \ &p<\infty, \\ \\
	\displaystyle \mathrm{ess} \sup_{x\in(0,1)}\|u(x_1,\cdot)\|_{L^q(0,G(x_1))}, \ &p=\infty.
	\end{cases}	
	\end{align*}
	
	When $p=q=2$ such space is Hilbert with the inner product
	\begin{equation*}
	(u,v)_{L^2(0,1;L^2(0,G(x_1))}=\displaystyle \int_0^1(u(x_1,\cdot),v(x_1,\cdot))_{L^2(0,G(x_1))}dx_1.
	\end{equation*}
\end{definition}

\begin{remark}
	Since $q<\infty$, the function $x_1\mapsto\|u(x_1,\cdot)\|_{L^q(0,G(x_1))}$ is measurable by Fubini's Theorem. Then the space $L^p(0,1;L^q(0,G(x_1))$ is well defined. 
\end{remark}

Analogously, the Sobolev-Bochner generalized spaces, denoted by $L^p(0,1;W^{s,q}(0,G(x_1)))$ for $s>0$, are defined by
\begin{equation*}
L^p(0,1;W^{s,q}(0,G(x_1))):=\{u\in L^p(0,1;L^q(0,G(x_1))); u(x_1,\cdot)\in W^{s,q}(0,G(x_1))\}.
\end{equation*}

Such spaces are Banach with the norm
\begin{align*}
\|u\|_{L^p(0,1;W^{s,q}(0,G(x_1)))}=\begin{cases}
\left(\displaystyle\int_0^1\|u(x_1,\cdot)\|_{W^{s,q}(0,G(x_1))}^pdx_1\right)^{1/p}, \ &p<\infty, \\ \\ \displaystyle \mathrm{ess} \sup_{x_1\in(0,1)}\|u(x_1,\cdot)\|_{W^{s,q}(0,G(x_1))}, \ &p=\infty,
\end{cases}
\end{align*}
and, again, they are Hilbert spaces if $p=q=2$.\\


In general, it follows from \cite[Proposition 3.59]{doina} that, if $H$ is a Hilbert space and $1\leq p<\infty$, then the dual space of $L^p(0,1;H)$ is given by 
$$[L^p(0,1;H)]' = L^q(0,1; H'),$$
where $H'$ is the dual space of $H$ and $p,q$ are conjugates. \\

	In our case we will consider the family of Lebesgue and Sobolev-Bochner generalized space for the function $G(x_1)=g_\vareps(x_1)$, with $g_\vareps$ as in Section \ref{resultado} (see Remark \ref{muh}).

\subsection{Some technical results} 

Next we will get some non-trivial properties which are important in our context.
First, we construct a unidimensional extension operator that will help us to work with different definitions of Sobolev fractional spaces, making their norms equivalent.

\begin{lemma}\label{ext_sob_uni}
	Fix $\vareps>0$ and $x_1\in(0,1)$, if we call $I_\vareps = (0,g_\vareps(x_1))$, with $g_\vareps$ as in Section \ref{resultado} (see Remark \ref{muh}), then there exists a continuous linear extension operator $P: L^2(I_\vareps)\to L^2(\R)$ such that $Pu=u$ in $I_\vareps$, with $\|Pu\|_{L^2(\R)}\leq \lambda_0\|u\|_{L^2(I_\vareps)}$, $\|Pu\|_{H^s(\R)}\leq \lambda_s\|u\|_{H^s(I_\vareps)}$ and $\|Pu\|_{H^1(\R)}\leq \lambda_1\|u\|_{H^1(I_\vareps)}$, for $0<s<1$, where the constants $\lambda_0,\lambda_s,\lambda_1\geq 1$ are independent of $\vareps>0$ and $x_1\in (0,1)$.
\end{lemma}
\begin{proof}
	Notice that $I_0 := (0,g_0)\subset I_\vareps$, for all $\vareps>0$. The construction of the extension operator will be in two steps: first we will extend the functions from $I_\varepsilon$ into $I=(0,g_1)$. Next, from interval $I$ into  $\R$.

If $2g_0\geq g_1$, we define $P_\vareps$ by a reflection procedure. If $\varphi\in L^2(I_\vareps)$,
\begin{equation*}
(P_\vareps\varphi)(y) = \begin{cases}\varphi(y), & \text{ if }y\in I_\vareps \\ 
\varphi(2g_\vareps(x_1)-y), & \text{ if }y\in I\smallsetminus I_\vareps \end{cases}.
\end{equation*}

Let us see that $P_\vareps$ 
is well defined, that is, that $(2g_\vareps(x_1)-y)\in I_\vareps$ if $y\in I\smallsetminus I_\vareps$. Indeed if $y\in I\smallsetminus I_\vareps$,
\begin{align*}\begin{cases}
2g_\vareps(x_1)-y>2g_\vareps(x_1)-g_1\geq 2g_\vareps(x_1)-2g_0>0\\2g_\vareps(x_1)-y<2g_\vareps(x_1)-g_\vareps(x_1) = g_\vareps(x_1)
\end{cases}\end{align*}
then $(2g_\vareps(x_1)-y)\in I_\vareps$. 

Now, let us show the continuity of the operator. If $\varphi\in H^1(I_\vareps)$,
\begin{align*}
\|P_\vareps\varphi\|^2_{L^2(I)} & = \|\varphi\|^2_{L^2(I_\vareps)}+\int_{2g_\vareps(x_1)-g_1}^{g_\vareps(x_1)} |\varphi(z)|^2dz \leq \|\varphi\|^2_{L^2(I_\vareps)}+\int_{2g_0-g_1}^{g_\vareps(x_1)} |\varphi(z)|^2dz \leq 2  \|\varphi\|^2_{L^2(I_\vareps)}.
\end{align*}
Besides
\begin{align*}
\left\|\ffrac{\partial P_\vareps\varphi}{\partial y}\right\|^2_{L^2(I)} & =  \int_{I_\vareps} \left|\ffrac{\partial \varphi}{\partial y}\right|^2+ \int_{g_\vareps(x_1)}^{g_1} \left|\ffrac{\partial}{\partial y} \varphi(2g_\vareps(x_1)-y)\right|^2 dy \\ & \leq  \left\|\ffrac{\partial \varphi}{\partial y}\right\|^2_{L^2(I_\vareps)}+\int_{2g_0-g_1}^{g_\vareps(x_1)} \left|\ffrac{\partial \varphi}{\partial y}(z)\right|^2dz \leq 2  \left\|\ffrac{\partial \varphi}{\partial y}\right\|^2_{L^2(I_\vareps)}.
\end{align*}

For $0<s<1$, we have
\begin{align*}
\|P_\vareps \varphi\|^2_{H^s(I)}&=\|P_\vareps \varphi\|^2_{L^2(I)}+\iint_{I\times I}\ffrac{|P_\varepsilon\varphi(y_1)-P_\varepsilon\varphi(y_2)|^2}{|y_1-y_2|^{1+2s}}dy \\ & \leq 2\|\varphi\|^2_{L^2(I_\vareps)} + \iint_{I_\vareps\times I_\vareps}\ffrac{|P_\varepsilon\varphi(y_1)-P_\varepsilon\varphi(y_2)|^2}{|y_1-y_2|^{1+2s}}dy +\iint_{A}\ffrac{|P_\varepsilon\varphi(y_1)-P_\varepsilon\varphi(y_2)|^2}{|y_1-y_2|^{1+2s}}dy \\ & \qquad +\iint_{B}\ffrac{|P_\varepsilon\varphi(y_1)-P_\varepsilon\varphi(y_2)|^2}{|y_1-y_2|^{1+2s}}dy+\iint_{C}\ffrac{|P_\varepsilon\varphi(y_1)-P_\varepsilon\varphi(y_2)|^2}{|y_1-y_2|^{1+2s}}dy \\ & \leq \|\varphi\|^2_{L^2(I_\vareps)} + \|\varphi\|^2_{H^s(I_\vareps)} + I_1+I_2+I_3,
\end{align*}
where
\begin{align*}
A = \{g_\vareps(x_1)< y_1<g_1\}\times\{g_\varepsilon(x_1)<y_2<g_1 \},\\
B = \{g_\vareps(x_1)< y_1<g_1\}\times\{0<y_2<g_\varepsilon(x_1) \},\\
C = \{0< y_1<g_\varepsilon(x_1)\}\times\{g_\varepsilon(x_1)<y_2<g_1 \}.
\end{align*} 

We analyze each integral separately. If we change variables as
$z_i=2g_\varepsilon(x_1)-y_i$, $i=1, 2$,
we get:
\begin{align*}
\iint_{A}\ffrac{|P_\varepsilon\varphi(y_1)-P_\varepsilon\varphi(y_2)|^2}{|y_1-y_2|^{1+2s}}dy & = \int_{g_\varepsilon(x_1)}^{g_1}\int_{g_\varepsilon(x_1)}^{g_1}\ffrac{|\varphi(2g_\varepsilon(x_1)-y_1)-\varphi(2g_\varepsilon(x_1)-y_2)|^2}{|y_1-y_2|^{1+2s}}dy \\ & \leq 2\iint_{I_\vareps\times I_\vareps}\ffrac{|\varphi(z_1)-\varphi(z_2)|^2}{|z_1-z_2|^{1+2s}}dz.
\end{align*}

On the other hand, if $(y_1,y_2)\in B$, then $y_2<g_\vareps(x_1)$, and if we call $z_1=2g_\varepsilon(x_1)-y_1$
\begin{align*}
\iint_{B}\ffrac{|P_\varepsilon\varphi(y_1)-P_\varepsilon\varphi(y_2)|^2}{|y_1-y_2|^{1+2s}}dy & = \int_{g_\varepsilon(x_1)}^{g_1}\int_{0}^{g_\varepsilon(x_1)}\ffrac{|\varphi(2g_\varepsilon(x_1)-y_1)-\varphi(y_2)|^2}{|y_1-y_2|^{1+2s}}dy \\ & = \iint_{I_\vareps\times I_\vareps}\ffrac{|\varphi(z_1)-\varphi(y_2)|^2}{|2g_\varepsilon(x_1)-z_1-y_2|^{1+2s}}dy_2dz_1\\ & \leq  3\iint_{I_\vareps\times I_\vareps}\ffrac{|\varphi(z_1)-\varphi(y_2)|^2}{|z_1-y_2|^{1+2s}}dy_2dz_1,
\end{align*}
since 
\begin{align*}
|z_1-y_2|&=|2g_\varepsilon(x_1)-y_1-y_2|= |y_2+y_1-2g_\varepsilon(x_1)| \\ & \leq |y_2-y_1|+2|y_1-g_\varepsilon(x_1)|\leq 3|y_1-y_2|=3|2g_\varepsilon(x_1)-z_1-y_2|.
\end{align*}

Analogously, we can show that 
\begin{align*}
\iint_{C}\ffrac{|P_\varepsilon\varphi(y_1)-P_\varepsilon\varphi(y_2)|^2}{|y_1-y_2|^{1+2s}}dy\leq 3\iint_{I_\vareps\times I_\vareps}\ffrac{|\varphi(y_1)-\varphi(z_2)|^2}{|y_1-z_2|^{1+2s}}dz_2dy_1
\end{align*}
proving
\begin{equation*}
\|P_\varepsilon\varphi\|_{H^s(I)}\leq C\|\varphi\|_{H^s(I_\vareps)}.
\end{equation*}

Now if $g_1>2g_0$, we first extend the initial function $\varphi$ in the direction of negative $y$ and then construct $P_\vareps$ in an analogous way to the previous one. In fact, if $\varphi_0$ is defined in $I_\vareps$, we can extend it to $\{y\in\R; \ -g_0<y<g_\vareps(x_1)\}$ as
\begin{equation*}
\varphi_1(y) = \begin{cases} \varphi_0(y), & \text{ if } 0<y<g_\vareps(x_1) \\   \varphi_0(-y), & \text{ if }  -g_0<y\leq 0 \end{cases}.
\end{equation*}

Iteratively, we can take 
\begin{equation*}
\varphi_n(y) = \begin{cases} \varphi_{n-1}(y), & \text{ if }  -(n-1)g_0<y<g_\vareps(x_1) \\   \varphi_{n-1}(-y-2(n-1)g_0), & \text{ if }  -ng_0<x_2\leq -(n-1)g_0 \end{cases}.
\end{equation*}

Thus, given $\varphi_0\in H^1(I_\vareps)$, there is $n$ sufficiently large such that $ng_0>g_1$, and then, we can define $P_\vareps$ as
\begin{equation*}
(P_\vareps\varphi_0)(y) = \begin{cases}\varphi_n(y), & \text{ if }y\in I_\vareps \\ 
\varphi_n(2g_\vareps(x_1)-y), & \text{ if }y\in I\smallsetminus I_\vareps \end{cases}.
\end{equation*}

It is not difficult to see that $P_\vareps$ is well defined.
Besides using $ng_0>g_1$, if $I_\varepsilon^n=(-ng_0,g_1)$, we have 
\begin{align*}
\|P_\vareps&\varphi_0\|^2_{L^2(I)} = \int_{I} |P_\vareps\varphi_0|^2 =  \int_{I_\vareps} |P_\vareps\varphi_0|^2 +  \int_{I\smallsetminus I_\vareps} |P_\vareps\varphi_0|^2 \\ & = \int_{I_\vareps} |\varphi_n|^p+ \int_{g_\vareps(x_1)}^{g_1} |\varphi_n(x_1,2g_\vareps(x_1)-y)|^2dy \leq    \int_{I_\vareps} |\varphi_n|^2+\int_{2g_\vareps(x)-g_1}^{g_\vareps(x_1)} |\varphi_n(y)|^2dy \\ & \leq 2 \|\varphi_n\|^2_{L^2(I_\vareps^n)} \leq 2(n+1)\|\varphi_0\|^2_{L^2(I_\vareps)}.
\end{align*}

In a similar way, we can perform the same estimate in $H^s(I)$ and $H^1(I)$, obtaining a extension operator from $I_\vareps$ into the interval $I$. 
Finally, we can set $P$ to the whole real line. 
Indeed, for $u\in H^1(I_\vareps)$, let $\psi\in C^\infty_c(\R)$ such that $I\subset supp(\psi)$, with $\psi=1$ in $I$ and $\psi=0$ in $\R\setminus (-g_0,g_1+g_0)$. 
Then we set
\begin{equation*}
Pu=\psi P_\varepsilon(u),
\end{equation*}
completing the proof.
\end{proof}

Now we state some properties of  Lebesgue and Sobolev-Bochner  generalized spaces that we will be needed in the analysis below.

\begin{proposition}\label{equiv_sob_norma_int2}
	Let $I_\vareps=(0,g_\vareps(x_1))$, with $\vareps>0$, $x_1\in (0,1)$ and $0<s<1$ fixed. Then there exist $C_1$, $C_2>0$ independent of $\vareps$ such that
	\begin{equation*}
	C_1\|u\|_{H^s(I_\vareps)}\leq \|u\|_{H^s_{[]}(I_\vareps)}\leq C_2\|u\|_{H^s(I_\vareps)}, \forall u\in H^s(I_\vareps),
	\end{equation*}
	where $H^s_{[]}(I_\vareps)$ is the complex interpolation space
	\begin{equation*}
	H^s_{[]}(I_\vareps)=[L^2(I_\vareps),H^1(I_\vareps)]_s, \ \text{ for } 0<s<1.
	\end{equation*}
\end{proposition}
\begin{proof}
	Using that there exists a continuous linear extension operator $P: L^2(I_\vareps)\to L^2(\R)$ given by Lemma \ref{ext_sob_uni}, if we define the space $\bar{H}^s(I_\vareps) = \{v _{|_{I_\vareps}} ; \ v\in H^s(\R)\}$, then for all $u\in H^s(I_\vareps)$ 
	\begin{equation*}
	K^{-1}\|u\|_{H^{s}(I_\vareps)}\leq\|u\|_{\bar{H}^s(I_\vareps)}\leq K\lambda_s \|u\|_{H^s(I_\vareps)}
	\end{equation*}
	where $K>0$ and $\|P\|_{\mathcal{L}(H^s(I_\vareps),H^s(\R))}\leq \lambda_s$ are independent of $\vareps$, with same notation from Lemma \ref{ext_sob_uni}. Indeed, if $u\in \bar{H}^s(I_\vareps)$, there is $U\in H^s(\R)$ such that $u=U_{|_{I_\vareps}}$ and $\|u\|_{\bar{H}^s(I_\vareps)}=\|U\|_{H(\R)}$. It follows that
	\begin{equation*}
	\|u\|_{H^{s}(I_\vareps)}\leq \|U\|_{H^s(\R)}=\|u\|_{\bar{H}^s(I_\vareps)}.
	\end{equation*}
	
	Reciprocally, if $u\in H^s(I_\vareps)$, then $u=U{|_{I_\vareps}}$ for $U=P u$. 
	It follows that $u\in \bar{H}^s(I_\vareps)$, with
	\begin{equation*}
	\|u\|_{\bar{H}^s(I_\vareps)}\leq \|U\|_{H^s(\R)}=\|P u\|_{H^s(\R)}\leq \lambda_s\|u\|_{H^s(I_\vareps)}.
	\end{equation*} 
	
	Analogously, by \cite[Lemma 4.2]{interpolation_sob} and Lemma \ref{ext_sob_uni}, we have
	\begin{equation*}
	\lambda_0^{s-1}\lambda_1^{-s}\|u\|_{\bar{H}^s(I_\vareps)}\leq \|u\|_{H^\theta_{[]}(I_\vareps)}\leq \|u\|_{\bar{H}^s(I_\vareps)}.
	\end{equation*}
	Then
	\begin{align*}
	\lambda_0^{s-1}\lambda_1^{-s}K^{-1}\|u\|_{H^{s}(I_\vareps)}&\leq \lambda_0^{s-1}\lambda_1^{-s}\|v\|_{\bar{H}^s(I_\vareps)}\leq \|u\|_{H^\theta_{[]}(I_\vareps)} \\ & \leq \|u\|_{\bar{H}^s(I_\vareps)} \leq K\lambda_s\|u\|_{H^s(I_\vareps)}
	\end{align*}
	proving the result for $C_1=\lambda_0^{s-1}\lambda_1^{-s}K^{-1}$ and $C_2=K\lambda_s$, with $0<s<1$.
\end{proof}

\begin{proposition}
	\label{leb_boch_int} For each $\vareps>0$, $H^1(\Omega_\vareps)\hookrightarrow L^2(0,1;H^s(0,g_\varepsilon(x_1)))$ for all $0\leq s\leq 1$, with embedding  constant independent of $\vareps$. Moreover, if   $0<s<1$, the embedding is compact. 
\end{proposition}
\begin{proof}
	For each $x_1\in (0,1)$ and $\vareps>0$, we have by Proposition \ref{equiv_sob_norma_int2} and properties of interpolation spaces that
	\begin{equation*}
	\|u(x_1,\cdot)\|_{H^s(0,g_\vareps(x_1))}\leq C \|u(x_1,\cdot)\|_{H_{[]}^s(0,g_\vareps(x_1))}\leq C\|u(x_1,\cdot)\|_{L^2(0,g_\vareps(x_1))}^{1-s}\|u(x_1,\cdot)\|^s_{H^1(0,g_\vareps(x_1))}
	\end{equation*}
	where $C>0$ is independent of $\vareps>0$ and $x_1\in(0,1)$.
	It follows that
	\begin{equation*}
	\|u(x_1,\cdot)\|_{H^s(0,g_\vareps(x_1))}\leq C\|u(x_1,\cdot)\|_{H^s(0,g_\vareps(x_1))}^{1-s}\|u(x_1,\cdot)\|^s_{H^1(0,g_\vareps(x_1))},
	\end{equation*}
	and then
	\begin{equation*}
	\|u(x_1,\cdot)\|^s_{H^s(0,g_\vareps(x_1))}\leq C\|u(x_1,\cdot)\|^s_{H^1(0,g_\vareps(x_1))}.
	\end{equation*}
	
	Consequently, if we integrate in $x_1\in(0,1)$ 
	\begin{align*}
	\|u\|^2_{L^2(0,1;H^s(0,g_\vareps(x_1)))}&=\int_0^1\|u(x_1,\cdot)\|^2_{H^s(0,g_\vareps(x_1))}dx_1\leq C\int_0^1\|u(x_1,\cdot)\|^2_{H^1(0,g_\vareps(x_1))}dx_1 \\ & = C\|u\|^2_{L^2(0,1;H^1(0,g_\vareps(x_1)))}\leq C\|u\|^2_{H^1(\Omega_\vareps)},
	\end{align*}
	concluding the first statement. To prove the last one, let us consider $Q=(0,1)^2$ and the function 
	\begin{align*}
	\phi_\vareps : & \ Q \to \Omega_\vareps : (x,y) \mapsto (x_1,x_2)=\phi_\vareps(x,y):=(x,yg_\vareps(x)).
	\end{align*}
	
	Consequently, we can set 
	\begin{align*}
	\Phi_\vareps : H^1&(\Omega_\vareps)\to H^1(Q) \\ & u \mapsto \Phi_\vareps(u):=u\circ \phi_\vareps
	\end{align*}
	and 
	\begin{align*}
	\Psi_\vareps : L^2(0,1;&H^s(0,1))\to L^2(0,1;H^s(0,g_\vareps(x_1))) \\ & u \mapsto \Psi_\vareps(u):=u\circ \phi^{-1}_\vareps.
	\end{align*}
	
	It is not difficult to see that $\Phi$ and $\Psi$ are continuous and satisfy  
	\begin{equation*}\label{aux11}
	\|\Phi_\vareps(u)\|_{H^1(Q)}\leq C_1\|u\|_{H^1(\Omega_\vareps)} 
	\quad \textrm{ and } \quad 
	\|\Psi_\vareps(u)\|_{L^2(0,1;H^s(0,g_\vareps(x_1)))}\leq C_2\|u\|_{L^2(0,1;H^s(0,1))}
	\end{equation*}
	for every $0<s<1$, and constants $C_1,C_2>0$ independents of $\vareps$.
	
	Hence, we can use \cite[Proposition 3.57]{doina} to obtain that the inclusion
	\begin{equation*}
	H^1(Q)\hookrightarrow L^2(0,1;H^s(0,1))
	\end{equation*}
	is compact.	
	Thus, we have the following chain
	\begin{equation*}
	H^1(\Omega_\vareps)\stackrel{\Phi_\vareps}{\hookrightarrow}H^1(Q)\hookrightarrow L^2(0,1;H^s(0,1))\stackrel{\Psi_\vareps}{\hookrightarrow}L^2(0,1;H^s(0,g_\vareps(x_1))),	
	\end{equation*}
	that implies the compact immersion.
\end{proof}

\subsection{Concentrated integrals}

Finally, we consider here what we call concentrated integrals. 

\begin{theorem}\label{lema_int_conc}
	For $\vareps_0>0$ sufficiently small, there is a constant $C>0$, independent of $\vareps\in(0,\vareps_0)$ and $u^\vareps\in H^1(\Omega_\vareps)$, such that, for all $1/2<s\leq1$,
	\begin{equation}\label{1}
	\ffrac{1}{\vareps}\int_{\theta_\vareps}|u^\vareps|^q\leq C \|u^\vareps\|^q_{L^q(0,1;H^s(0,g_\varepsilon(x_1)))}, \quad  \forall q\geq 1,
	\end{equation}
	and
	\begin{align}\label{2}
	\ffrac{1}{\vareps}\int_{\theta_\vareps}|u^\vareps|^2\leq  C \left(\|u^\vareps\|_{H^s(\Omega_\vareps)}^{2}+\left\|\ffrac{\partial u^\vareps}{\partial x_2}\right\|^2_{L^2(\Omega_\vareps)}\right).
	\end{align}
	
	In particular,
	\begin{equation}\label{3}
	\ffrac{1}{\vareps}\int_{\theta_\vareps}|u^\vareps|^2\leq C \|u^\vareps\|^2_{H^1(\Omega_\vareps)}.
	\end{equation}
\end{theorem}
\begin{proof}	
	Take $u\in H^1(\Omega_\vareps)$. In a.e. $x_1\in(0,1)$, we have $u(x_1,\cdot)\in H^1(0,g_\varepsilon(x_1))$. Define
	\begin{equation*}
	z^*:=g_0-\vareps_0 h_1 \text{ and } z^\vareps:=g_\varepsilon(x_1)-\vareps h_\vareps(x_1)
	\end{equation*}
	for $\vareps_0>0$ sufficiently small in such way that, for all $\vareps<\vareps_0$, we have
	\begin{equation*}
	[z^\vareps-z^*,z^\vareps]\subset[0,g_\varepsilon(x_1)].
	\end{equation*}
	
	\begin{figure}[!h]
		\centering\includegraphics[scale=0.46]{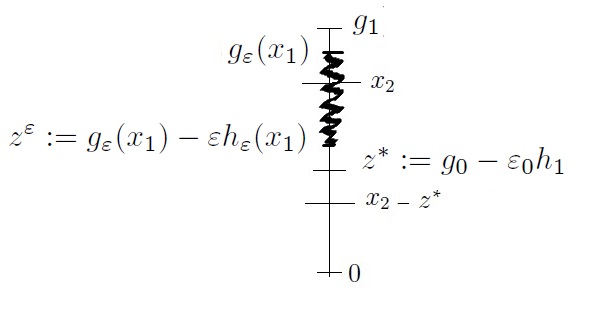}
		\caption[Fiber of the oscillatory domain]{Fixed $x_1\in(0,1)$ and $\varepsilon>0$, we get this fiber to the oscillatory domain for $\varepsilon<\varepsilon_0$.}
		\label{fig:lema}
	\end{figure}
	
	Since $(g_\varepsilon(x_1)-\vareps h_\vareps(x_1))<x_2<g_\varepsilon(x_1)$ and $1/2<s\leq1$, it follows from \cite[Theorem 1.5.1.3]{grisvard} for $n=1$ that exists $K>0$ independent of $\vareps>0$ such that
	\begin{equation*}
	|u(x_1,x_2)|\leq K\|u(x_1,\cdot)\|_{H^s(x_2-z^*,x_2)} \leq K\|u(x_1,\cdot)\|_{H^s(0,g_\varepsilon(x_1))}.
	\end{equation*}
	Indeed, the interval where we are applying the result is fixed and independent of the parameters $\vareps$ and $x_1$.
	
	Hence,
	\begin{align*}
	\ffrac{1}{\vareps}\int_{\theta_\vareps}|u|^q&=\int_{0}^{1}\ffrac{1}{\vareps}\int_{g_\varepsilon(x_1)-\vareps h_\vareps(x_1)}^{g_\varepsilon(x_1)}|u(x_1,x_2)|^q dx_2dx_1 \\ & \leq \int_{0}^{1}\ffrac{1}{\vareps}\int_{g_\varepsilon(x_1)-\vareps h_\vareps(x_1)}^{g_\varepsilon(x_1)}K^q\|u(x_1,\cdot)\|_{H^s(0,g_\varepsilon(x_1))}^q dx_2dx_1 \\ & \leq K^q h_1\int_{0}^{1}\|u(x_1,\cdot)\|_{H^s(0,g_\varepsilon(x_1))}^q dx_1 = C_1\|u\|^q_{L^q(0,1;H^s(0,g_\varepsilon(x_1)))},
	\end{align*}
	where $C_2$ is independent of $\vareps$, proving \eqref{1}.
	
	If $q=2$, due to Proposition \ref{leb_boch_int} and previous inequality, we get 
	\begin{equation*}
	\ffrac{1}{\vareps}\int_{\theta_\vareps}|u|^2\leq C_1\|u\|^2_{L^2(0,1;H^s(0,g_\varepsilon(x_1)))}  \leq C_1C_2\|u\|^2_{H^1(\Omega_{\vareps})}
	\end{equation*}
	proving \eqref{3}.
	
	Now, let us prove \eqref{2}. Here we use that $C^\infty(\Omega_\vareps)$ is dense in $H^1(\Omega_\vareps)$ (see \cite[Theorem 1.4.2.2]{grisvard}).
	Let $u\in C^\infty(\Omega_\vareps)$ and fixed $x_1\in(0,1)$. By Fundamental Theorem of Calculus, we have  
	\begin{equation*}
	u(x_1,x_2)=u(x_1,0)+\int_{0}^{x_2}\ffrac{\partial u}{\partial x_2}(x_1,s)ds.
	\end{equation*}
	
	Then
	\begin{align*}
	|u(x_1,x_2)|^2 &\leq 2|u(x_1,0)|^2+2\left[\left(\int_{0}^{x_2}\left|\ffrac{\partial u}{\partial x_2}(x_1,s)\right|^2ds\right)^{1/2}\left(\int_{0}^{x_2}1^2ds\right)^{1/2}\right]^2 \\ &\leq  2|u(x_1,0)|^2+2g_\varepsilon(x_1)\int_{0}^{x_2}\left|\ffrac{\partial u}{\partial x_2}(x_1,s)\right|^2ds.
	\end{align*}
	
	Consequently,
	\begin{align*}
	\int_{g_\varepsilon(x_1)-\vareps h(x_1,\vareps)}^{g_\varepsilon(x_1)}&|u(x_1,x_2)|^2 dx_2\leq 2\int_{g_\varepsilon(x_1)-\vareps h_\vareps(x_1)}^{g_\varepsilon(x_1)}|u(x_1,0)|^2dx_2\\ & \qquad \quad \qquad +2g_\varepsilon(x_1)\int_{g_\varepsilon(x_1)-\vareps h_\vareps(x_1)}^{g_\varepsilon(x_1)}\left(\int_{0}^{x_2}\left|\ffrac{\partial u}{\partial x_2}(x_1,s)\right|^2ds\right)dx_2 \\ & \leq	2\vareps h_1|u(x_1,0)|^2+2g_1\vareps h_1\int_{0}^{g_\varepsilon(x_1)}\left|\ffrac{\partial u}{\partial x_2}(x_1,x_2)\right|^2dx_2.
	\end{align*}
	
	Hence, if $\gamma(u)$ is the trace of $u$ given by \cite[Theorem 1.5.1.3]{grisvard}, we get 
	\begin{align*}
	\ffrac{1}{\vareps}\int_{\theta_\vareps}|u|^2& = \ffrac{1}{\vareps}\int_{0}^{1}\int_{g_\varepsilon(x_1)-\vareps h_\vareps(x_1)}^{g_\varepsilon(x_1)}|u(x_1,x_2)|^2 dx_2dx_1 \\ & \leq 2 h_1\int_{0}^{1}|u(x_1,0)|^2dx_1 +2g_1h_1\int_{0}^{1}\int_{0}^{g_\varepsilon(x_1)}\left|\ffrac{\partial u}{\partial x_2}(x_1,x_2)\right|^2dx_2dx_1 \\ & \leq 2h_1\left(\|\gamma(u)\|_{L^2(0,1)}^{2}+g_1\left\|\ffrac{\partial u}{\partial x_2}\right\|^2_{L^2(\Omega_\vareps)}\right).
	\end{align*}
	
	On the other hand, if $\Omega_0=(0,1)\times(0,g_0)$, we have $\Omega_0\subset\Omega_\vareps$, and there exists a constant $c>0$ such that $\|\gamma(u)\|_{L^2(0,1)}\leq c\|u\|_{H^s(\Omega_0)}$ for all $1/2<s\leq1$. Then, due to the previous inequality,
	\begin{align*}
	\ffrac{1}{\vareps}\int_{\theta_\vareps}|u|^2\leq  2h_1\left(c\|u\|_{H^s(\Omega_0)}^{2}+g_1\left\|\ffrac{\partial u}{\partial x_2}\right\|^2_{L^2(\Omega_\vareps)}\right)\leq C_1 \left(\|u\|_{H^s(\Omega_\vareps)}^{2}+\left\|\ffrac{\partial u}{\partial x_2}\right\|^2_{L^2(\Omega_\vareps)}\right)
	\end{align*}
	with $C_1$ independent of $\vareps$. 
\end{proof}

\section{Nonlinearities} \label{nonlinear}

In this section, we show some properties to a class of nonlinear maps defined in Sobolev-Bochner spaces. Such applications will define the nonlinearity of our elliptic problems. 

Consider the Sobolev-Bochner spaces 
\begin{equation}
X_\vareps = L^2(0,1;H^s(0,g(x_1,\vareps))), \text{ and their dual } X'_\vareps = L^2(0,1; \{H^{s}(0,g(x_1,\vareps)) \}' ), 
\label{x_e}
\end{equation} 

for  $1/2<s<1$, and define
\begin{align}\label{f_e}
F_\vareps : X_\vareps&\to X'_\vareps \nonumber\\ u&\mapsto  F_\vareps(u) : X_\vareps\to \R  \\ & \quad \qquad \qquad v\mapsto\langle F_\vareps(u),v\rangle = \ffrac{1}{\vareps}\int_{\theta_\vareps}f(u)v\nonumber
\end{align}
where $f\in C^2(\R)$ is a bounded function with bounded derivatives (see Remark \ref{limit_f}).
Thus, we have:

\begin{proposition}\label{F_dif1} The function $F_\vareps$ defined in \eqref{f_e} satisfies, with constants independents of $\vareps$:
	\begin{enumerate}[(a)]
		\item there is $K>0$ such that
		\begin{equation*}
		\sup_{u^\varepsilon\in X_\vareps}\|F_\varepsilon(u^\vareps)\|_{X'_\vareps}\leq K;
		\end{equation*}
		
		\item $F_\varepsilon$ is Lipschitz and, therefore, is continuous; in other words there is $L>0$ such that
		\begin{equation*}
		\|F_\varepsilon(u_1^\vareps)-F_\varepsilon(u_2^\vareps)\|_{X'_\vareps}\leq L\|u_1^\vareps-u_2^\vareps\|_{X_\vareps}, \ \forall u_1,u_2\in X_\vareps;
		\end{equation*}
	\end{enumerate}
\end{proposition}
\begin{proof}
	\begin{enumerate}[(a)]
		\item For $u^\varepsilon\in X_\vareps$,
		\begin{equation*}
		\|F_\varepsilon(u^\vareps)\|_{X_\vareps'}=\sup_{\|v^\vareps\|_{X_\vareps}=1}|\langle F_\varepsilon(u^\vareps),v^\varepsilon\rangle|.
		\end{equation*}
		
		So, if $v^\varepsilon\in X_\vareps$, using Theorem \ref{lema_int_conc} we have
		\begin{align*}
		|\langle F_\varepsilon(u^\vareps),v^\varepsilon\rangle|&\leq\ffrac{1}{\vareps}\int_{\theta_\vareps}|f(u^\vareps)v^\vareps|\leq \left(\ffrac{1}{\vareps}\int_{\theta_\vareps}|f(u^\vareps)|^2\right)^{1/2}\left(\ffrac{1}{\vareps}\int_{\theta_\vareps}|v^\vareps|^2\right)^{1/2} 
		\leq \|f\|_{\infty}h_1^{1/2}C\|v^\vareps\|_{X_\vareps}.
		\end{align*}
		
		Therefore
		\begin{equation*}
		\sup_{u^\varepsilon\in X_\vareps}\|F_\varepsilon(u^\vareps)\|_{X_\vareps'}\leq \|f\|_{\infty}h_1^{1/2}C \leq K
		\end{equation*}
		
		\item Indeed, if $u_1^\vareps,u_2^\vareps\in X_\vareps$ then
		\begin{align*}
		\|F_\vareps(u^\vareps_1)-F_\vareps(u_2^\vareps)\|_{X_\vareps'}= \sup_{\|v^\vareps\|_{X_\vareps}=1}|\langle F_\vareps(u^\vareps_1),v^\vareps\rangle-\langle F_\vareps(u^\vareps_2),v^\vareps\rangle|
		\end{align*}
		Using Theorem \ref{lema_int_conc},
		\begin{align*}
		|\langle F_\vareps(u_1^\vareps),v^\vareps\rangle-\langle &F_\vareps(u_2^\vareps),v^\vareps\rangle|=|\langle F_\vareps(u_1^\vareps)-F_\vareps(u_2^\vareps),v^\vareps\rangle| \leq \ffrac{1}{\vareps}\int_{\theta_\vareps}|(f(u_1^\vareps)-f(u_2^\vareps))v^\vareps|\\ & \leq \left(\ffrac{1}{\vareps}\int_{\theta_\vareps}|f(u_1^\vareps)-f(u_2^\vareps)|^2\right)^{1/2}\left(\ffrac{1}{\vareps}\int_{\theta_\vareps}|v^\vareps|^2\right)^{1/2}\\ & \leq \|f'\|_{\infty}C^2\|u_1^\vareps-u_2^\vareps\|_{X_\vareps}\|v^\vareps\|_{X_\vareps}
		\end{align*}
		Thus
		\begin{equation*}
		\|F_\vareps(u_1^\vareps)-F_\vareps(u_2^\vareps)\|_{X_\vareps'}\leq \|f'\|_{\infty}C^2\|u_1^\vareps-u_2^\vareps\|_{X_\vareps}
		\end{equation*}
		\noindent and, therefore, $F_\vareps$ is Lipschitz with constant independent of $\vareps$.
	\end{enumerate}
\end{proof}

\section{Upper and lower semicontinuity} \label{mainsec}

In this section, we prove the main result passing to the limit in problem \eqref{inicio}.
First, we write equations \eqref{limite} and \eqref{inicio} in an abstract way. Next, we combine the results from the previous sections with those ones from \cite{arrieta_simone_lips,dumb_1} concerned with compact convergence to obtain upper and lower semicontinuity to $\E_{\vareps,R}$ at $\vareps=0$.   

\subsection{Abstract setting and existence of solutions}

In order to write problem \eqref{inicio} in an abstract way, we consider the linear operator
\begin{align*}
A_\vareps &: D(A_\vareps)\subset L^2(\Omega_\vareps)\to L^2(\Omega_\vareps)\\ &\qquad \qquad u^\vareps\longmapsto A_\vareps u^\vareps:= -\ffrac{\partial^2 u^\vareps}{\partial x_1^2}-\ffrac{1}{\vareps^2}\ffrac{\partial^2 u^\vareps}{\partial x_2^2}+u^\vareps,
\end{align*}
with $D(A_\vareps):=\left\{u^\vareps\in H^2(\Omega_\vareps); \ \ffrac{\partial u^\vareps}{\partial x_1}N_1+\ffrac{1}{\vareps^2}\ffrac{\partial u^\vareps}{\partial x_2}N_2=0\right\}$.

Let $Z_\vareps^0=L^2(\Omega_\vareps)$, $Z_\vareps^1=D(A_\vareps)$ and consider the scale of Hilbert spaces $Z_\vareps^\alpha$ constructed by complex interpolation between $Z_\vareps^0$ and $Z_\vareps^1$. In our context, such spaces isometrically coincide with the fractional power space $A_\vareps^\alpha$ of the operator $A_\vareps$ (see \cite[Theorem 16.1]{yagi}). Such scale can be extended to negative exponents taking $Z^{-\alpha}_\vareps = (Z_\varepsilon^\alpha)'$ for $\alpha>0$. Notice that $Z_\vareps^{1/2}=H^1_\vareps(\Omega_{\vareps})$ 
and $Z_\varepsilon^{-1/2}=(H^1_\vareps(\Omega_\vareps))'$
where $H^1_\vareps(\Omega_{\vareps})$ is the space $H^1(\Omega_{\vareps})$ endowed with the equivalent norm
\begin{equation*}
\|u^\vareps\|^2_{H^1_\vareps(\Omega_{\vareps})}=\|u^\vareps\|^2_{L^2(\Omega_{\vareps})}+\left\|\ffrac{\partial u^\vareps}{\partial x_1}\right\|^2_{L^2(\Omega_{\vareps})}+\ffrac{1}{\vareps^2}\left\|\ffrac{\partial u^\vareps}{\partial x_2}\right\|^2_{L^2(\Omega_{\vareps})}.
\end{equation*}

Then, if we consider the realizations of $A_\vareps$ in this scale, we obtain $A_{\varepsilon,-1/2}\in \mathcal{L}(Z_\varepsilon^{1/2},Z_\varepsilon^{-1/2})$ with
\begin{equation*}
\langle A_{\varepsilon,-1/2} \ u^\vareps,\varphi^\vareps \rangle = \int_{\Omega_\vareps}\ffrac{\partial u^\vareps}{\partial x_1}\ffrac{\partial \varphi^\vareps}{\partial x_1}+\frac{1}{\vareps^2}\ffrac{\partial u^\vareps}{\partial x_2}\ffrac{\partial \varphi^\vareps}{\partial x_2}  + u^\varepsilon\varphi^\varepsilon, \quad \forall \varphi^\vareps\in H^1(\Omega_\vareps).
\end{equation*} 
With some abuse of notation, we identify all different realizations of this operator writing them as $A_\vareps$. Then  problem \eqref{inicio} can be rewrite as
\begin{equation}\label{caso_se_a}
A_\vareps u^\varepsilon=F_\varepsilon(u^\vareps),
\end{equation}
where the map $F_\varepsilon$ is given by
\begin{align*}
F_\vareps : L^2(0,1;H^s(0,g_\vareps(x)))&\to L^2(0,1; \{ H^{s}(0,g_\vareps(x)) \}') \nonumber\\ u^\vareps&\mapsto  F_\vareps(u^\vareps) : L^2(0,1;H^s(0,g_\vareps(x)))\to \R  \\ & \qquad \qquad \qquad v^\vareps\mapsto\langle F_\vareps(u^\vareps),v^\vareps\rangle = \ffrac{1}{\vareps}\int_{\theta_\vareps}f(u^\vareps)v^\vareps,\nonumber
\end{align*}
with $1/2<s<1$.

Thus $u^\varepsilon\in H^1(\Omega_\vareps)$ is a solution of \eqref{caso_se_a} if, and only if, $u^\varepsilon=A_\varepsilon^{-1}F_\varepsilon(u^\vareps)$. Then $u^\vareps\in H^1(\Omega_\vareps)$ must be a fixed point to $A_\varepsilon^{-1}F_\varepsilon|_{H^1(\Omega_{\vareps})} : H^1(\Omega_\vareps)\to H^1(\Omega_\vareps)$. The existence of such solutions follows from Schaefer Fixed Point Theorem \cite[Section 9.2.2, Theorem 4]{evans}.

In a similar way, we can analyze the limit problem given by \eqref{limite}.
We first consider $X_0 = L^2(0,1)$ with the norm $\|u\|^2_{X_0} = \mu_g\|u\|^2_{L^2(0,1)}$ and, then, the linear operator
\begin{align*}
A_0 : D(&A_0)\subset X_0\to X_0 \\ u &\mapsto A_0 u = -q_0u_{xx} + u,
\end{align*}
with $D(A_0) = \{u\in H^2(0,1);u'(0)=u'(1)=0 \}$. Next, we introduce the fractional power spaces, and set the nonlinearity  
\begin{align*}
F_0 : X_0&\to L^2(0,1) \nonumber\\ u&\mapsto  F_0(u) : L^2(0,1)\to \R  \\ & \qquad \qquad \qquad v\mapsto\langle F_0(u),v\rangle = \int_0^1\mu_hf_0(u)dx,
\end{align*}
where $q_0$ and $f_0$ are given in \eqref{q_0}.

Consequently, the limit problem \eqref{limite} can be rewritten as
\begin{equation}\label{caso_se_l}
A_0 u=F_0(u)
\end{equation}
and then, $u\in H^1(\Omega)$ is a solution of \eqref{caso_se_l} if, and only if, $u=A_0^{-1}F_0(u)$. Thus, $u\in H^1(0,1)$ is a fixed point to $A_0^{-1}F_0 : H^1(0,1)\to H^1(0,1)$. The existence of a solution also follows from Schauder's Fixed Point Theorem.

\subsection{Extension operator}

Now, we consider a continuous extension linear operator that will be useful in our situation.
More precisely, from \cite[Lemma 3.1]{arr_marcone_artigo_base} we have

\begin{lemma}
	\label{lema_ext}
	If $\Omega_\vareps = \{(x_1,x_2); \ 0<x_1<1, \ 0<x_2<g_\vareps(x_1)\}$ and $\Omega = \{(x_1,x_2); \ 0<x_1<1, \ 0<x_2<g_1\}$, then there exist a constant $K>0$, independent of $\vareps$ and $p$, and an extension operator
	$$P_\vareps\in \mathcal{L}(L^p(\Omega_\vareps),L^p(\Omega))\cap \mathcal{L}(W^{1,p}(\Omega_\vareps),W^{1,p}(\Omega)) \cap \mathcal{L}(W^{1,p}_{\partial_l}(\Omega_\vareps),W^{1,p}_{\partial_l}(\Omega))$$
	\noindent(where $W^{1,p}_{\partial_l}$ is the set of functions in $W^{1,p}$ that vanish in the domain's lateral boundary) such that
	\begin{align*}
	\|P_\vareps\varphi^\vareps\|_{L^p(\Omega)}&\leq K \|\varphi^\vareps\|_{L^p(\Omega_\vareps)}, \\
	\left\|\ffrac{\partial P_\vareps\varphi^\vareps}{\partial x_1}\right\|_{L^p(\Omega)}\leq K &\left(\left\|\ffrac{\partial \varphi^\vareps}{\partial x_1}\right\|_{L^p(\Omega_\vareps)} +\eta(\vareps)\left\|\ffrac{\partial \varphi^\vareps}{\partial x_2}\right\|_{L^p(\Omega_\vareps)}\right), \\
	\left\|\ffrac{\partial P_\vareps\varphi^\vareps}{\partial x_2}\right\|_{L^p(\Omega)}&\leq K \left\|\ffrac{\partial \varphi^\vareps}{\partial x_2}\right\|_{L^p(\Omega_\vareps)},
	\end{align*}
	\noindent for all $\varphi^\vareps\in W^{1,p}(\Omega_\vareps)$, with $1\leq p \leq\infty$ and $\eta(\vareps)=\sup{|g'_\vareps(x_1)}|$.
\end{lemma}

This operator will play an important role in the convergence analysis since it extends the functions defined in the perturbed domain $\Omega_\vareps$ into the fixed one $\Omega$ in an appropriate way. One important property of this extension operator is the following.

\begin{proposition}\label{prop_norma}
 If $\|u\|_{H^1_\vareps(\Omega_\vareps)}\leq K$, with $K>0$ independent of $\vareps$, then $\|P_\vareps u\|_{H^1(\Omega)}$ is uniformly bounded and we can extract a subsequence (still denoted by $\vareps$) such that 
 \begin{equation*}
	P_\vareps u^\vareps\to u_0 \text{ in } L^2(\Omega), \quad P_\vareps u^\vareps\rightharpoonup u_0 \text{ in } H^1(\Omega), \quad \ffrac{\partial P_\vareps u^\vareps}{\partial x_2}\to 0 \text{ in } L^2(\Omega),
	\end{equation*}
	\noindent for some $u_0\in H^1(0,1)$, where $P_\varepsilon$ is the extension operator from Lemma \ref{lema_ext}.
In particular, if $1/2 < s < 1$,
$$
P_\vareps u^\vareps \to u_0, \quad \textrm{ in } X = L^2(0,1;H^s(0,g_1)).
$$ 
\end{proposition}
\begin{proof}
	In fact, since $\|u^\varepsilon\|_{H^1_\varepsilon(\Omega_\vareps)}\leq C$ we have
	\begin{equation*}
	\|u^\vareps\|_{L^2(\Omega)} \leq C, \quad  \left\|\ffrac{\partial u^\vareps}{\partial x_1}\right\|_{L^2(\Omega)} \leq C  \text{ and } \left\|\ffrac{\partial u^\vareps}{\partial x_2}\right\|_{L^2(\Omega)} \leq C\vareps.
	\end{equation*}
	
	Using Lemma \ref{lema_ext},
	\begin{align*}
	\|P_\vareps u^\vareps\|_{H^1(\Omega)}^2&=\|P_\vareps u^\vareps\|_{L^2(\Omega)}^2 + \left\|\ffrac{\partial P_\vareps u^\vareps}{\partial x_1}\right\|^2_{L^2(\Omega)}+\left\|\ffrac{\partial P_\vareps u^\vareps}{\partial x_2}\right\|^2_{L^2(\Omega)} \\ &\leq K\left[ \|u^\vareps\|^2_{L^2(\Omega_\vareps)} + \left(\left\|\ffrac{\partial u^\vareps}{\partial x_1}\right\|^2_{L^2(\Omega_\vareps)} +\eta^2(\vareps)\left\|\ffrac{\partial u^\vareps}{\partial x_2}\right\|^2_{L^2(\Omega_\vareps)}\right) + \left\|\ffrac{\partial u^\vareps}{\partial x_2}\right\|^2_{L^2(\Omega_\vareps)}\right]	
	\end{align*}
	
	Since $g$ has bounded derivative, it follows
	\begin{equation*}
	\eta(\vareps)=\ffrac{1}{\varepsilon}\sup|g'(x_1/\vareps)|\leq \ffrac{C}{\vareps},
	\end{equation*}
	and then, for $0<\vareps<1$ there is $M>0$ independent of $\varepsilon>0$ such that
	\begin{align*}
	\|P_\vareps u^\vareps\|_{H^1(\Omega)} \leq \bar{C}\|u^\vareps\|^2_{H^1_\vareps(\Omega_\vareps)}\leq M.
	\end{align*}

	Consequently,
	\begin{align*}
	\|P_\vareps u^\vareps\|_{L^2(\Omega)}\leq M, \ \left\|\ffrac{\partial P_\vareps u^\vareps}{\partial x_1}\right\|_{L^2(\Omega)} \leq M, \  \left\|\ffrac{\partial P_\vareps u^\vareps}{\partial x_2}\right\|_{L^2(\Omega)}\leq \vareps M.
	\end{align*}
	
	Thus, there exist a subsequence $P_\vareps u^\varepsilon$ and a function $u_0\in H^1(\Omega)$ such that
	\begin{equation*}
	P_\vareps u^\vareps\to u_0 \text{ in } L^2(\Omega), \quad P_\vareps u^\vareps\rightharpoonup u_0 \text{ in } H^1(\Omega), \quad \ffrac{\partial P_\vareps u^\vareps}{\partial x_2}\to 0 \text{ in } L^2(\Omega).
	\end{equation*}
	Further, from Proposition \ref{leb_boch_int}, we have $P_\vareps u^\vareps \to u_0$ in $X=L^2(0,1;H^s(0,g_1))$.  
	
	It follows that $u_0(x_1,x_2) = u_0(x_1)$ in $\Omega$, in particular, 
	\begin{equation*}
	\ffrac{\partial u_0}{\partial x_2}(x_1,x_2)=0 \ \text{ a.e. in } \Omega.
	\end{equation*}
	Indeed, for all $\varphi\in C_0^\infty(\Omega)$, we have
	\begin{align*}
	\int_{\Omega}\ffrac{\partial u_0}{\partial x_2}&\varphi dx_1dx_2 = -\int_{\Omega}u_0\ffrac{\partial \varphi}{\partial x_2} dx_1dx_2 = 
	- \lim_{\vareps\to 0}\int_{\Omega}(P_\vareps u^\vareps)\ffrac{\partial \varphi}{\partial x_2} dx_1dx_2  = - \lim_{\vareps\to 0}\int_{\Omega}\ffrac{\partial P_\vareps u^\vareps}{\partial x_2}\varphi dx_1dx_2=0.
	\end{align*}
	Furthermore, since $u_0\in H^1(\Omega)$ and $u_0(x,y)$ does not depend on the second variable, we may regard $u_0(x,y)=u_0(x)$ and  $u_0\in H^1(0,1)$, concluding the proof.
\end{proof}

\subsection{Continuity of the equilibria set}
In this section, we show our main result. First, let us collect some lemmas and propositions in order to achieve our goal.

\begin{lemma}\label{lema5.2}
	Let $u^\vareps\in H^1(\Omega_\vareps)$ and denote by $w^\vareps\in H^1(\Omega_{\vareps})$  the function $w^\vareps= A_\vareps^{-1}F_\vareps(u^\vareps)$. 
	Then $\|w^\vareps\|_{H^1(\Omega_{\vareps})}\leq C$ for some $C>0$ independent of $\vareps$.
	\end{lemma}
\begin{proof}
	Since $w^\vareps= A_\vareps^{-1}F_\vareps u^\vareps$, it follows that, for any $\varphi\in H^1(\Omega_\vareps)$,
	\begin{equation*}
	\int_{\Omega_\vareps}\ffrac{\partial w^\vareps}{\partial x_1}\ffrac{\partial\varphi}{\partial x_1}+\int_{\Omega_\vareps}\ffrac{\partial w^\vareps}{\partial x_2}\ffrac{\partial \varphi}{\partial x_2}+\int_{\Omega_\vareps}w^\vareps\varphi = \ffrac{1}{\vareps}\int_{\theta_{\vareps}}f(u^\vareps)\varphi
	\end{equation*}
	Therefore, if we take $\varphi=w^\vareps$, it follows from Lemma \ref{lema_int_conc} that 
	\begin{align*}
	\|w^\vareps\|^2_{H^1(\Omega_\vareps)}\leq \left( \ffrac{1}{\vareps}\int_{\theta_{\vareps}}|f(u^\vareps)|^2\right)^{1/2}&\left( \ffrac{1}{\vareps}\int_{\theta_{\vareps}}|w^\vareps|^2\right)^{1/2}\leq \|f\|_\infty h_1^{1/2}\|w^\vareps\|_{H^1(\Omega_{\vareps})}.
	\end{align*}
	Thus, $\|w^\vareps\|_{H^1(\Omega_\vareps)}\leq C$.
\end{proof}

Now, we will analyze the asymptotic behavior of the nonlinearities. 

\begin{proposition}\label{int_conv1}
	If we have $u^\varepsilon\in H^1(\Omega_\varepsilon)$ and $u_0\in H^1(\Omega)$ with $P_\vareps u^\vareps \rightharpoonup u_0$ in $H^1(\Omega)$ and $u_0(x_1,x_2)=u_0(x_1)$, for all $(x_1,x_2)\in \Omega$, then, for all $\varphi\in H^1(0,1)$, we have
	\begin{equation}
	\label{lim_eq_conc}
	\ffrac{1}{\vareps}\int_{\theta_\vareps}f(u^\vareps)\varphi \to \mu_h\int_0^1 f(u_0)\varphi
	\end{equation}
	\noindent where $\mu_h$ is the average of the function $h$ as in Remark \ref{muh}.
\end{proposition}
\begin{proof}  
	Indeed, 
	\begin{align*}
	\ffrac{1}{\vareps}\int_{\theta_\vareps}f(u^\vareps)\varphi - \mu_h\int_0^1 f(u_0)\varphi = \ffrac{1}{\vareps}\int_{\theta_\vareps}(f(u^\vareps)-f(u_0))\varphi + \ffrac{1}{\vareps}\int_{\theta_\vareps}f(u_0)\varphi- \int_0^1 \mu_hf(u_0)\varphi.	
	\end{align*}
	Using the definition of $\mu_h$ and with standard computations, we have 
	\begin{align*}
	&\left|\ffrac{1}{\vareps}\int_{\theta_\vareps}f(u^\vareps)\varphi - \mu_h\int_0^1 f(u_0)\varphi\right| \leq \ffrac{1}{\vareps}\int_{\theta_\vareps}|f(u^\vareps(x_1,x_2))-f(u_0(x_1))||\varphi(x_1)|dx_2dx_1 \\ & \ \ \ 
	+  \left|\ffrac{1}{\vareps}\int_0^1\int_{g_\vareps(x_1)-\vareps h_\vareps(x_1)}^{g_\vareps(x_1)}f(u_0(x_1))\varphi(x_1)dx_2dx_1 - \int_0^1 f(u_0(x_1))\mu_h\varphi(x_1)dx_1\right| \\ & \leq \left(\ffrac{1}{\vareps}\int_{\theta_\vareps}|f(u^\vareps(x_1,x_2))-f(u_0(x_1))|^2dx_2dx_1\right)^{1/2}\left(\ffrac{1}{\vareps}\int_{\theta_\vareps}|\varphi(x_1)|^2dx_2dx_1\right)^{1/2} \\ & \ \ \ 
	+ \left|\int_0^1 \ffrac{1}{\vareps}[g_\vareps(x_1)-(g_\vareps(x_1)-\vareps h_\vareps(x_1))]f(u_0(x_1))\varphi(x_1)dx_1 - \int_0^1 f(u_0(x_1))\mu_h\varphi(x_1)dx_1\right|
	\end{align*}
	
Now,  since $\varphi=\varphi(x_1)$ we have $\left|\ffrac{1}{\vareps}\int_{\theta_\vareps}|\varphi(x_1)|^2dx_2dx_1\right|\leq h_1\|\varphi\|_{L^2(0,1)}^2\leq C\|\varphi\|_{H^1(0,1)}^2$.  Moreover, using Theorem \ref{lema_int_conc}, Remark  \ref{muh}  and the uniform bound of $f$ and $f'$, it follows that, for $1/2<s<1$,
	\begin{align*}
	 & \leq \|f'\|_{\infty}C_1\left(\ffrac{1}{\vareps}\int_{\theta_\vareps}|u^\vareps(x_1,x_2)-u_0(x_1)|^2dx_2dx_1\right)^{1/2}\|\varphi\|_{H^1(0,1)} 
	+  \left|\int_0^1 f(u_0(x_1))\varphi(x_1)(h_\vareps(x_1)-\mu_h)dx_1\right| \\ & \leq K\|u^\vareps-u_0\|_{X_\vareps}+  \left|\int_0^1 f(u_0(x_1))\varphi(x_1)(h_\vareps(x_1)-\mu_h)dx_1\right| \\ & \leq \tilde K \|P_\vareps u^\vareps-u_0\|_{X} +  \left|\int_0^1 f(u_0(x_1))\varphi(x_1)(h_\vareps(x_1)-\mu_h)dx_1\right| \longrightarrow 0 \text{ as } \vareps\to 0.
	\end{align*}
	This completes the proof.
\end{proof}


We also need a notion of compactness for sequences, and convergence for operators which are defined in different spaces.
We follow the exposition from \cite{CP}. See also \cite{arrieta_simone_lips}. 

In general, consider a family of Hilbert spaces $X_\vareps$ and a limit Hilbert space $X_0$. Besides, let $E_\vareps : X_0\to X_\vareps$ a family of operators such that $\|E_\vareps u\|_{X_\vareps}\to\|u\|_{X_0}$ when $\vareps \to 0$. 
We recall that a sequence $u^\vareps\in X_\varepsilon$ $E$-converges to $u_0\in X_0$, if $\|u^\vareps-E_\vareps u\|_{X_\varepsilon}\to0$. This will be denoted by $u_\vareps\xrightarrow{E}u$.

\begin{definition}
	A sequence $\{u_n\}$, $u_n\in X_{\varepsilon_n}$ with $\vareps_n\to0$, is $E$-precompact if for all subsequence $\{u_{n'}\}$ there are a subsequence $\{u_{n''}\}$ and an element $u\in X_0$ such that $u_{n''}\xrightarrow{E}u$. A family is said to be $E$-precompact is all sequence $\{u_n\}$, $u_n\in X_{\varepsilon_n}$ with $\vareps_n\to0$, is $E$-precompact.
\end{definition}

\begin{definition}
	We say that a family of operators $\{T_\vareps\}$, with $T_\vareps : X_\varepsilon \to X_\varepsilon$, $E$-converges to $T : X_0 \to X_0$ when $\vareps\to0$ if $T_\vareps u^\vareps\xrightarrow{E} Tu$ for any $u^\vareps \xrightarrow{E} u$. We denote this convergence by $T_\vareps\xrightarrow{EE}T$. 
\end{definition}

Finally, we may define a notion of compact convergence for operators.
\begin{definition}
	A family of compact operators $\{T_\vareps\}$, with $T_\vareps : X_\varepsilon \to X_\varepsilon$, converges compactly to $T : X_0 \to X_0$ when $\vareps\to0$ if, for any family $\{u^\vareps\}$ with $\|u^\vareps\|_{H^1(\Omega_\vareps)}$ uniformly bounded, we have that $\{T_\vareps u^\vareps\}$ is $E$-precompact and $T_\vareps \xrightarrow{EE}T$. We denote this compact convergence by $T_\vareps\xrightarrow{CC}T$. 
\end{definition}


For now on, consider again the spaces $X_\vareps$ defined in \eqref{x_e}. The next result show the compact convergence of the operators $A_\vareps^{-1}$ to $A_0^{-1}$, defined in \eqref{caso_se_a}, in the Sobolev-Bochner generalized spaces $L^2(0,1;H^s(0,g_\vareps(x_1)))$. 

\begin{proposition} \label{passo1}
	Using the notation given by \eqref{caso_se_a} and \eqref{caso_se_l}, we have 
 $A_\vareps^{-1} \stackrel{CC}{\longrightarrow} A_0^{-1}$ with $A_\vareps^{-1} : X_\vareps \to X_\vareps$.
\end{proposition}
\begin{proof}
	It will be proved in three parts.
	\begin{enumerate}[(i)]
		\item $A^{-1}_\vareps$ is compact for each $\vareps>0$.
		
		Using previous results from, for instance \cite{arr_marcone_artigo_base}, we have $A^{-1}_\vareps:L^2(\Omega_\vareps)\to H^1(\Omega_\vareps)$ is compact. Hence, since $X_\vareps$ is continuously embedded in $L^2(\Omega_\vareps)$, and $H^1(\Omega_\vareps)$ is compactly embedded in $X_\vareps$ by Proposition \ref{leb_boch_int}, we have that 
		\begin{equation*}
		X_\vareps \stackrel{i}{\longrightarrow} L^2(\Omega_\vareps) \stackrel{A_\vareps^{-1}}{\hookrightarrow} H^1(\Omega_\vareps) \stackrel{i}{\longrightarrow} X_\vareps.
		\end{equation*} 
		Thus, $A^{-1}_\vareps:X_\vareps\to X_\vareps$ is a family of compact operators for each $\vareps>0$. The proof for $\vareps=0$ is analogous.
		
		\item The family $\{A^{-1}_\vareps f^\vareps\}$ is $E$-precompact when $\|f^\vareps\|_{X_\vareps}$ is bounded.
		
		In fact, if $\{f^\vareps\}_{\vareps\in(0,1)}$ in $X_\vareps$ is such that $\|f^\vareps\|_{X_\vareps}\leq M$, define $u^\vareps:=A^{-1}_\vareps f^\vareps$. Then $A_\vareps u^\vareps=f^\vareps$, and $u^\vareps$ satisfies, for $\vareps$ sufficiently small, that
		\begin{align*}
		\|u^\vareps\|^2_{H^1_\vareps(\Omega_\vareps)} & \leq \int_{\Omega_\vareps}|f^\vareps u^\vareps|\leq \left(\int_{\Omega_\vareps}|f^\vareps|^2\right)^{1/2}\left(
		\int_{\Omega_\vareps}|u^\vareps|^2\right)^{1/2} \\ & = \|f^\vareps\|_{L^2(\Omega_\vareps)} \|u^\vareps\|_{L^2(\Omega_\vareps)}  \leq \|f^\vareps\|_{X_\vareps}\|u^\vareps\|_{H^1(\Omega_\vareps)} \leq M\|u^\vareps\|_{H^1_\vareps(\Omega_\vareps)}.
		\end{align*}
		
		Using Proposition \ref{prop_norma}, it follows that $\|P_\vareps u^\vareps\|_{H^1(\Omega)}$ is uniformly bounded and there are $u_0\in H^1(0,1)$ and subsequence, that we will also call $P_\vareps u^\vareps$, such that $P_\vareps u^\vareps\rightharpoonup u_0$ in $H^1(\Omega)$ and, consequently, $P_\vareps u^\vareps \to u_0$ in $X$. Furthermore it follows that
		\begin{equation*}
		\|A^{-1}_\vareps f^\vareps-E_\vareps u_0\|_{X_\vareps} =\|u^\vareps-E_\vareps u_0\|_{X_\vareps}=\|(P_\vareps u^\vareps-u_0)|_{\Omega_\vareps}\|_{X_\vareps}\leq \|P_\vareps u^\vareps-u_0\|_{X}\to 0.
		\end{equation*} 
		
		\item If $f_\vareps\stackrel{E}{\rightarrow} f_0$, then $A_\vareps^{-1}f^\vareps \stackrel{E}{\rightarrow} A_0^{-1}f_0$.
		
		Indeed, like the previous item, suppose $u^\vareps:=A_\vareps^{-1} f^\vareps$. It follows that $A_\vareps u^\vareps=f^\vareps$, $\|f^\vareps\|_{X_\vareps}$ is bounded, since is $E$-convergent, and then there are again subsequence of $u^\vareps$ (also called $u^\vareps$) and $u_0\in H^1(0,1)$ such that $P_\vareps u^\vareps \rightharpoonup u_0$ in $H^1(\Omega)$. 
		
		Since $f_\vareps \stackrel{E}{\rightarrow} f_0$, it follows that
		\begin{equation*}
		\hat{f^\vareps} := \int_{0}^{g_\vareps(x_1)}f^\vareps(x_1,x_2)dx_2\rightharpoonup f_0
		\end{equation*}
		in $L^2(0,1)$. Thus, using \cite[Theorem 4.3]{arr_marcone_artigo_base} we have that $u_0\in H^1(0,1)$ satisfies
		\begin{equation*}
		\int_{0}^{1}(-q_0u''_0+u_0)\varphi=\int_{0}^{1}f_0\varphi, \ \forall \varphi \in H^1(0,1).
		\end{equation*}
		Furthermore, $A_0u_0=f_0$, that is, $u_0=A_0^{-1}f_0$ and
		\begin{equation*}
		\|A^{-1}_\vareps f^\vareps-E_\vareps A_0^{-1}f_0\|_{X_\vareps} =\|u^\vareps-E_\vareps u_0\|_{X_\vareps}=\|(P_\vareps u^\vareps-u_0)|_{\Omega_\vareps}\|_{X_\vareps}\leq \|P_\vareps u^\vareps-u_0\|_{X}\to 0.
		\end{equation*}  
	\end{enumerate}
	Thus, we conclude the proof.
\end{proof}

Now, we prove the result that will guarantee the semicontinuity of the solutions of problem \eqref{caso_se_a}.
\begin{proposition}\label{cc1_sim}
	Using the previous notation,  we have $A_\vareps^{-1}F_\vareps\xrightarrow{CC}A_0^{-1}F_0$ as $\varepsilon \to 0$.
\end{proposition}
\begin{proof}
	We will also divide this proof in three parts.
\begin{enumerate}[(a)]
	\item $A^{-1}_\vareps F_\vareps$ is compact for each $\vareps>0$, where $A_\vareps^{-1} F_\vareps : X_\vareps\to X_\vareps$.
	
	In fact, since $A_\vareps^{-1}$ is compact by Proposition \ref{passo1} and $F_\vareps$ is Lipscthiz with constant that is independent of $\vareps$ by Proposition \ref{F_dif1}$(b)$, the result follows by composing those applications. Thus $A^{-1}_\vareps F_\vareps:X_\vareps\to X_\vareps$ is a family of compact operators for each $\vareps>0$. The proof for $\vareps=0$ is analogous.
		
	\item $\{A^{-1}_\vareps F_\vareps(u^\vareps)\}$ is $E$-precompact when $\|u^\vareps\|_{X_\vareps}$ is bounded. 
	
	Define $z^\vareps:=A^{-1}_\vareps F_\vareps(u^\vareps)$ and, consequently, $A_\vareps z^\vareps= F_\vareps(u^\vareps)$. Since $u^\vareps\in X_\vareps$, we have $z^\vareps\in H^1(\Omega_\vareps)$. If we call $f^\vareps = f(u^\vareps)\chi^{\theta_{\vareps}}/\vareps$ for each $\vareps>0$, we obtain $f^\vareps\in L^2(\Omega_\vareps)$ and
	\begin{equation*}
	\int_{\Omega_{\vareps}}|f^\vareps|^2 = \ffrac{1}{\vareps}\int_{\theta_\vareps} |f(u^\vareps)|^2\leq \|f\|_{\infty}^2h_1 = K
	\end{equation*}
	\noindent where $K>0$ is independent of $\vareps$. It follows that $\|z^\vareps\|_{H^1_\vareps(\Omega_\vareps)}\leq K$ and by Proposition \ref{prop_norma} we have that $\|P_\vareps z^\vareps\|_{H^1(\Omega_\vareps)}$ is uniformly bounded, where $P_\vareps$ is the extension operator from Lemma \ref{lema_ext}. 
	
	Also, from Proposition \ref{prop_norma}, there are $z_0\in H^1(\Omega)$ and subsequence, that we will also call $P_\vareps z^\vareps$, such that $P_\vareps z^\vareps \to z_0$ in $X$ and $z_0$ is independent of the second variable. Hence, we have $z_0 \in H^1(0,1)\subset X_0$.
	
	Thus,
	\begin{equation*}
	\|A^{-1}_\vareps F_\vareps(u^\vareps)-E_\vareps z_0\|_{X_\vareps} = \|z^\vareps-E_\vareps z_0\|_{X_\vareps}=\|(P_\vareps z^\vareps-z_0)|_{\Omega_\vareps}\|_{X_\vareps}\leq \|P_\vareps z^\vareps-z_0\|_{X}\to0.
	\end{equation*}
	
	\item $A_\vareps^{-1}F_\vareps(u^\vareps) \xrightarrow{E} A_0^{-1}F_0(u_0)$ if $u^\vareps\xrightarrow{E} u_0$.
	
	Arguing as in the previous item, let us define $z^\vareps:=A^{-1}_\vareps F_\vareps(u^\vareps)$, and then, $A_\vareps z^\vareps= F_\vareps(u^\vareps)$. Since $u^\vareps\in X_\vareps$, we have $z^\vareps\in H^1(\Omega_\vareps)$. If we call $f^\vareps = f(u^\vareps)\chi^{\theta_{\vareps}}/\vareps$ for each $\vareps>0$, we have
	\begin{equation*}
	\int_{\Omega_{\vareps}}|f^\vareps|^2dx = \ffrac{1}{\vareps}\int_{\theta_\vareps} |f(u^\vareps)|^2\leq \|f\|_\infty^2 h_1= K,
	\end{equation*}
	\noindent with $K>0$ independent of $\vareps$. Furthermore, since $u^\vareps \xrightarrow{E}u_0$, if we define 
	\begin{equation*}
	\hat{f^\vareps}(x_1)=\ffrac{1}{\vareps}\int_{g_\vareps(x_1)-\vareps h_\vareps(x_1)}^{g_\vareps(x_1)}f(u^\vareps(x_1,x_2))dx_2,
	\end{equation*}
	we have $\hat{f^\vareps} \rightharpoonup \hat{f}$ in $L^2(0,1)$ by Proposition \ref{int_conv1}, with $\hat{f}(x_1) = \mu_hf(u_0(x_1))$. Indeed, for all $\varphi\in L^2(0,1)$ 
	\begin{align*}
	\int_0^1&\left(\ffrac{1}{\vareps}\int_{g_\vareps(x_1)-\vareps h_\vareps(x_1)}^{g_\vareps(x_1)}f(u^\vareps) dx_2-\mu_h f(u_0)\right) \varphi dx_1= \ffrac{1}{\vareps}\int_{\theta_\vareps}f(u^\vareps)\varphi- \int_0^1\mu_hf(u_0)\varphi \to 0.
	\end{align*}
	Consequently, from \cite[Theorem 4.3]{arr_marcone_artigo_base} there is $z_0\in H^1(0,1)$ such that $P_\vareps z^\vareps \to z_0$ in $X$, where $z_0$ satisfies, for all $\varphi\in H^1(0,1)$,
	\begin{equation*}
	\int_0^1(q_0z_0'\varphi'+z_0\varphi) = \int_0^1 \ffrac{L_g}{|Y^*|}\mu_gf(u_0)\varphi.
	\end{equation*}
	It follows from the definition of $A_0$ and $F_0$ that $z_0=A_0^{-1}F_0(u_0)$, and then,
	\begin{align*}
	\|A^{-1}_\vareps F_\vareps(u^\vareps)-&E_\vareps A_0^{-1}F_0(u_0)\|_{X_\vareps} = \|z^\vareps-E_\vareps z_0\|_{H^s(\Omega_\vareps)}\\&=\|(P_\vareps z^\vareps-z_0)|_{\Omega_\vareps}\|_{X_\vareps}\leq \|P_\vareps z^\vareps-z_0\|_{X}\to0
	\end{align*}
	concluding the proof.
	\end{enumerate}
\end{proof}

As a consequence of Proposition \ref{cc1_sim}, we can get the following proposition:
\begin{proposition}\label{semi_sup_simone}
	For any family of solutions $\{u^\vareps_*\}$ of \eqref{caso_se_a}, there is $u_*$ solution of \eqref{caso_se_l} and a subsequence of $u^\vareps_*$, also called $u^\vareps_*$, such that $u^\vareps_*\xrightarrow{E}u_*$.
\end{proposition}
\begin{proof}
	It is a direct consequence of \cite[Corollary 5.2]{arrieta_simone_lips} or \cite[Proposition 5.6]{dumb_1}.
\end{proof}

We also get the reciprocal of the previous proposition when the limit solution is hyperbolic.

\begin{proposition}\label{semi_inf_simone}
	If the solution $u^*$ of \eqref{caso_se_l} is hyperbolic, then there is a sequence $\{u^\vareps_*\}$ of solutions of \eqref{caso_se_a} such that $u^\vareps_*\xrightarrow{E}u^*$.
\end{proposition}
\begin{proof}
	It follows from \cite[Corollary 5.3]{arrieta_simone_lips} or \cite[Proposition 5.7]{dumb_1}.
\end{proof}

\begin{remark}
	In the case when all equilibria points from the limit equation \eqref{caso_se_l} are hyperbolic, we have that all of them are isolated, and then, there exists only a finite number of them (see \cite[Corollary 5.4 or Proposition 5.5]{dumb_1}). 
\end{remark}

Furthermore, the previous results prove the upper and lower semicontinuity of the equilibrium set at $\vareps=0$.
\begin{proof}[Proof of Theorem \ref{teo_main}]
The item $(a)$ follows from Proposition \ref{semi_sup_simone}, and the item $(b)$ is a consequence of Proposition \ref{semi_inf_simone}.
\end{proof}

\vspace{0.7 cm}

{\bf Acknowledgements.} 
The first author (JMA)$^*$ is partially  supported by grants MTM2016-75465,  ICMAT Severo Ochoa project SEV-2015-0554, MINECO, Spain and Grupo de Investigaci\'on CADEDIF, UCM. The second author (AN)$^\diamond$ was supported by CNPq 141869/2013-5, Brazil. Third one (MCP)$^\dagger$ is partially supported by CNPq 303253/2017-7 and FAPESP 2017/02630-2 Brazil.

\end{document}